\documentclass[12pt, a4paper,reqno]{amsart}
\usepackage{amscd,amssymb,amsmath, array,latexsym,amsbsy}
\usepackage[active]{srcltx}
\def\Prim{\operatorname{Prim}}
\def\Glimm{\operatorname{Glimm}}
\def\Min{\operatorname{Min}}
\def\Primal{\operatorname{Primal}}
\def\Orc{\operatorname{Orc}}
\def\Sub{\operatorname{Sub}}

\def\C{\mathbb{C}}
\def\T{\mathbb{T}}
\def\Z{\mathbb{Z}}
\def\R{\mathbb{R}}

\def\NN{\mathbb{N}}

\newtheorem{thm}{Theorem}[section]
\newtheorem{cor}[thm]{Corollary}
\newtheorem{prop}[thm]{Proposition}
\newtheorem{lemma}[thm]{Lemma}
\theoremstyle{definition}

\numberwithin{equation}{section}
\title[Norms of inner derivations for multiplier algebras]{Norms of inner derivations for multiplier algebras of $C^*$-algebras
and group $C^*$-algebras, II.}

\thanks{The authors are grateful to the London Mathematical Society
for grant number 4919, which partially supported a research visit by
E. Kaniuth to the University of Aberdeen, and they are also grateful to the referee for a number of helpful comments.}

\begin{document}

\maketitle

 \centerline{Robert J. Archbold$^a$, Eberhard Kaniuth$^b$
and Douglas W. B. Somerset$^c$}

\bigskip

$^a$Institute of Mathematics, University of Aberdeen, King's
College, Aberdeen AB24 3UE, U.K. Email: r.archbold@abdn.ac.uk

\noindent Corresponding author. Tel. +44 1224 272756. Fax +44 1224
272607.

\medskip

$^b$Institute of Mathematics, University of Paderborn, 33095
Paderborn, Germany. \newline \noindent Email:
kaniuth@math.uni-paderborn.de

\medskip

$^c$Institute of Mathematics, University of Aberdeen, King's
College, Aberdeen AB24 3UE, U.K. Email:
somerset@quidinish.fsnet.co.uk

\bigskip

\noindent {\bf Abstract.} The derivation constant $K(A)\geq
\frac{1}{2}$ has been extensively studied for \emph{unital}
non-commutative $C^*$-algebras. In this paper, we investigate
properties of $K(M(A))$ where $M(A)$ is the multiplier algebra of a
non-unital $C^*$-algebra $A$. A number of general results are
obtained which are then applied to the group $C^*$-algebras
$A=C^*(G_N)$ where $G_N$ is the motion group $\R^N\rtimes SO(N)$.
Utilising the rich topological structure of the unitary dual
$\widehat{G_N}$, it is shown that, for $N\geq3$,
$$K(M(C^*(G_N)))= \frac{1}{2}\left\lceil \frac{N}{2}\right\rceil.$$


\bigskip

\thanks{}


\bigskip


\bigskip

\section{Introduction.}



\bigskip

 For a C$^*$-algebra $A$, an elementary application of the
triangle inequality shows that
$$\Vert D(a,A)\Vert\leq 2 d(a, Z(A))$$ for all $a\in A$, where $D(a,A)$ is the inner
derivation generated by $a$ and $d(a, Z(A))$ is the distance from
$a$ to $Z(A)$, the centre of $A$.  This leads naturally to the
definition of $K(A)$  as the smallest number in $[0,\infty]$ such
that
$$K(A)\Vert D(a,A)\Vert\ge d(a, Z(A))$$ for all $a\in A$ \cite{RJA, KLR}. If the
elements $a$ are restricted to be self-adjoint then the
corresponding constant is denoted by $K_s(A)$. If $A=B(H)$ (or, more
generally, a non-commutative von Neumann algebra on a Hilbert space
$H\neq\C$) then $K(A)=\frac{1}{2}$ \cite{Stam, Zs}.
For unital non-commutative $C^*$-algebras,
$K_s(A)=\frac{1}{2}\Orc(A)$ \cite{Som}, where the \emph{connecting
order} $\Orc(A)\in\NN\cup\{\infty\}$ is determined by a graph
structure in the primitive ideal space ${\rm Prim}(A)$ (see Section
2),  and for the constant $K(A)$ it has been shown that the only
possible positive values less than or equal to
$\frac{1}{2}+\frac{1}{\sqrt 3}$ are:
$$\frac{1}{2}, \quad\frac{1}{\sqrt
3}, \quad1, \quad \frac{3+8\sqrt 2}{14}, \quad \frac{4}{\sqrt 15},
\quad \frac{1}{2}+\frac{1}{\sqrt 3}$$ \cite{Some, Id, ASom}. These
results use the fine structure of the topology on $\Prim(A)$
together with spectral constructions and the constrained
optimization of the bounding radii of planar sets.

\medskip

If $A$ is a non-unital $C^*$-algebra then, as discussed in
\cite{AKSI}, the multiplier algebra $M(A)$ is the natural
unitization to consider in the context of inner derivations.
For example, it is well-known that if $A$ is a primitive $C^*$-algebra then so is $M(A)$ (cf. \cite[Example 5.5]{AKSI}) and so $K(M(A))=\frac{1}{2}$ \cite[Theorem 5]{Stam}. In particular, $K(M(A))=\frac{1}{2}$ for every simple $C^*$-algebra $A$.

\medskip

In general, in order to apply to $M(A)$ the results for unital algebras, there is a
prima facie requirement for more detailed information on $\Prim(M(A))$.
However, this space is usually much larger and more complicated than
the dense open subset $\Prim(A)$. This is illustrated by the
complexity of the Stone-\v{C}ech compactification $\beta\NN$ of the
natural numbers $\NN$ and also by the results in \cite{ASidealsX},
which apply to the motion group $C^*$-algebras considered in this
paper (see the remarks after Theorem~\ref{Orc(G_N)}). However, when
$A$ is $\sigma$-unital, the normality of the complete regularization
of $\Prim(A)$ enables ideal structure in $M(A)$ to be linked to
ideal structure in $A$ without having full knowledge of
$\Prim(M(A))$ (Proposition~\ref{Prop-Glimm-normal}). It follows from
this that, in several cases of interest, the value of $K(M(A))$ is
determined by the ideal structure in $A$ itself and hence by the
topological properties of the $T_0$-space $\Prim(A)$
\cite[3.1]{Dix}.
 This allows the possibility of computing
$K(M(A))$ for $A=C^*(G)$ in cases where $G$ is a locally compact
group whose unitary dual $\widehat{G}$ is well-understood as a
topological space.
In \cite{AKSI}, we obtained two general $C^*$-theoretic results for
$K(M(A))$ which enabled us to show that
$$K(M(C^*(G)))=K_s(M(C^*(G)))=1$$ for a number of well-known locally
compact groups $G$, including $SL(2,\R)$, $SL(2,\C)$ and the
classical motion group of the plane $G_2 = \R^2\rtimes SO(2)$.

\medskip

In this paper, we focus on $C^*(G_N)$ where $G_N$ is the
 motion group $\R^N\rtimes
SO(N)$ ($N\geq 3$). Since $G_N$ is a Type I group, ${\rm
Prim}(C^*(G_N))$ is homeomorphic to the unitary dual
$\widehat{G_N}$, which is known to have a rich topological structure
\cite{Bag, KST, EL}. Indeed, the $C^*$-algebra  $C^*(G_N)$ has
recently been identified, via the Fourier transform, with an
explicit algebra of operator fields over $\widehat{G_N}$ \cite{AEL}.
We use the topological structure of $\widehat{G_N}$ in showing that,
for $N\geq3$,
$$K(M(C^*(G_N)))=
K_s(M(C^*(G_N)))= \frac{1}{2}\left\lceil \frac{N}{2}\right\rceil.$$
Somewhat surprisingly, this formula is not valid for the case $N=2$,
despite the fact that $C^*(G_2)$ is the most well-behaved of the
motion group $C^*$-algebras (by \cite{KST} it is the only one which
is quasi-standard in the sense of \cite{ASo}).
 In contrast to the results above, note that it follows from
\cite[Proposition 2.1]{AKSI} that $K(C^*(G_N))=K_s(C^*(G_N))=1$ for
all $N\geq2$. Thus $K(M(A))$ gives much more information than $K(A)$
for the algebras $A=C^*(G_N)$.

\medskip

   In the course of determining the values for $K(M(C^*(G_N)))$,
we obtain in Sections 4, 5 and 7 several new results for general
$C^*$-algebras $A$, sometimes under the assumption that $A$ is
$\sigma$-unital and that the complete regularization map on
$\Prim(A)$ is closed. For example, in
Theorem~\ref{Orc(M(A))geqOrc(A)} we give sufficient conditions
(which are satisfied in the case $A=C^*(G_N)$) for the inequality
$$\Orc(A)\leq \Orc(M(A)).$$
 Section 5 gives some general upper bounds
for $K(M(A))$, including a new result for unital $C^*$-algebras
(Theorem~\ref{unital(n+1/2)}). The inequality $\Orc(A)\leq
\Orc(M(A))$ from Section 4 then combines with the results from
Section 5 to determine $K(M(C^*(G_N)))$ in the case where $N$ is
even (Section 6).

\medskip

When $N$ is odd, it turns out that we need a sharper estimate for
$\Orc(M(C^*(G_N)))$. Accordingly, in Section 7, we introduce a new
constant $D(A)$ arising from a graph structure on $\Sub(A)$, a
subset of the set of primal ideals of a $C^*$-algebra $A$. This is
closely linked to the way in which $\Orc(A)$ is obtained from the
graph structure on $\Prim(A)$. Indeed, either
$|\Orc(A)-D(A)|\in\{0,1\}$ or $\Orc(A)=D(A)=\infty$. In
Theorem~\ref{Orc(M(A))=D(A)+1}, we give sufficient conditions (which
are satisfied by $A=C^*(G_N)$) for $N\geq3$) for the equality
$$\Orc(M(A))=D(A)+1.$$
This exemplifies our earlier contention that, when $A$ is
$\sigma$-unital, ideal structure in $M(A)$ can be usefully related
to ideal structure in $A$.
 This equality then combines with the results from
Section 5 to determine $K(M(C^*(G_N)))$ in the case where $N$ is odd
(Section 8).


 \bigskip

\section{Preliminaries.}

\bigskip

We begin by recalling some terminology from \cite{Som}. Let $X$ be a
topological space. For $x, y\in X$ we write $x\sim y$ if $x$ and $y$
cannot be separated by disjoint open sets. The relation $\sim$ is
reflexive and symmetric but it is not always transitive. We will
view $X$ as a graph in which two points $x$ and $y$ are adjacent if
and only if $x\sim y$. For $x, y\in X$ let $d(x,y)$ denote the
distance from $x$ to $y$ in the graph $(X, \sim)$. If there is no
walk from $x$ to $y$ we write $d(x,y)=\infty$. We define the {\sl
diameter} of a $\sim$-connected component of $X$ to be the supremum
of the distances between pairs of points in the component, except
that we adopt the non-standard convention that the diameter of a
singleton component is $1$ (rather than $0$). Define $\Orc(X)$, the
connecting order of $X$, to be the supremum of the diameters of
$\sim$-connected components of $X$. By virtue of our non-standard
convention, $\Orc(X)=1$ when $X$ is a Hausdorff space. In the case
when $X$ is the primitive ideal space of a C$^*$-algebra $A$ we
write $\Orc(A)$ instead of $\Orc(\Prim(A))$; and sometimes we write
$d_A$, in place of $d$, for the distance function when we need to
emphasize the algebra we are working in. If $\sim$ is an open
equivalence relation on $\Prim(A)$ (that is, $\sim$ is an
equivalence relation and the corresponding quotient map is open)
then the $C^*$-algebra $A$ is said to be \emph{quasi-standard} (see
\cite{ASo}, where several equivalent conditions and examples are
given). Note that if $A$ is quasi-standard then $\Orc(A)=1$.

It was shown in \cite[Theorem 4.4]{Som} that, if $A$ is a unital
$C^*$-algebra, $K_s(A)= \frac{1}{2}\Orc(A)$. It follows that if $A$
is any $C^*$-algebra then $K_s(M(A))= \frac{1}{2}\Orc(M(A))$ and so
$\frac{1}{2}\Orc(M(A))\leq K(M(A))$. It turns out that equality
holds in the case where $A=C^*(G_N)$ ($N\geq3$). We shall show this
by establishing, for the four cases modulo $4$, that
$$K(M(A))\leq\frac{1}{2}\left\lceil
\frac{N}{2}\right\rceil\qquad\hbox{ and also }\qquad\left\lceil
\frac{N}{2}\right\rceil\leq\Orc(M(A)).$$

\bigskip

We now recall some properties of the complete regularization of
$\Prim(A)$ for a $C^*$-algebra $A$ (see \cite{DH} for further
details). For $P,Q\in\Prim(A)$ let $P\approx Q$ if and only if
$f(P)=f(Q)$ for all $f\in C^b(\Prim(A))$. Then $\approx$ is an
equivalence relation on $\Prim(A)$ and the equivalence classes are
closed subsets of $\Prim(A)$. It follows that there is a one-to-one
correspondence between $\Prim(A)/\approx$ and a set of closed
two-sided ideals of $A$ given by
$$[P]\to \bigcap[P]\qquad (P\in\Prim(A)),$$
where $\bigcap[P]$ is the intersection of the ideals in the equivalence class $[P]$ of $P$. The set of ideals
obtained in this way is denoted by $\Glimm(A)$ and we identify this
set with $\Prim(A)/\approx$ by the correspondence above. If $A$ is
unital then $\Glimm(A)$ consists of the ideals of $A$ generated by
the maximal ideals of the centre of $A$, as studied by Glimm
\cite{Glimm}. The quotient map $\phi_A:\Prim(A)\to\Glimm(A)$ is
called \emph{the complete regularization map}. The standard topology
on $\Glimm(A)$ is the topology $\tau_{cr}$, which is the weakest
topology for which the functions on $\Glimm(A)$ induced by
$C^b(\Prim(A))$ are all continuous. This topology is completely
regular, Hausdorff, weaker than the quotient topology (and equal to
it when $A$ is $\sigma$-unital \cite[Theorem 2.6]{Lazar}) and hence
makes $\phi_A$ continuous. The ideals in $\Glimm(A)$ are called
\emph{Glimm ideals} and the equivalence classes for $\approx$ in
$\Prim(A)$ will sometimes be referred to as \emph{Glimm classes}.

Note that if $P,Q\in \Prim(A)$, $G\in \Glimm(A)$ and $P\supseteq
G=\bigcap[Q]$ then, since $[Q]$ is closed, $P\in [Q]$ and so
$\phi_A(P)=\phi_A(Q)=G$. It follows that, for $P\in\Prim(A)$ and
$G\in\Glimm(A)$, $P\supseteq G$ if and only if $\phi_A(P)=G$. For
$P,Q\in\Prim(A)$, it is clear that $P\sim Q$ implies that $P\approx
Q$. The converse implication holds whenever $A$ is quasi-standard
\cite[Proposition 3.2]{ASo}. In general, a Glimm class is said to be
\emph{$\sim$-connected} if it consists of a single $\sim$-component.

\medskip

We recall that $A$ is said to be \emph{$\sigma$-unital} if it
contains a strictly positive element or, equivalently, a countable
approximate unit \cite[3.10.5]{GKP}.  If $A$ is $\sigma$-unital with
a strictly positive element $u$ then $\Prim(A)$ is the union of the
compact sets $\{P\in \Prim(A):\Vert u+P\Vert \geq 1/n\}$ ($n\geq1$).
Since $\phi_A$ is continuous, Glimm(A) is $\sigma$-compact, hence
Lindel\"{o}f, and therefore normal by (complete) regularity  (see
\cite[3D]{GJ} or \cite[Ch.2, Proposition 1.6]{Pears}).

\medskip

There is a homeomorphism $\iota$ from $\beta\Glimm(A)$ onto
$\Glimm(M(A))$ such that $\iota(\phi_A(P))=\phi_{M(A)}(\tilde{P})$
where, for $P\in\Prim(A)$, $\tilde{P}$ is the unique primitive ideal
of $M(A)$ such that $\tilde{P}\cap A=P$ (see, for example, \cite[p.
88]{MM} and \cite[Proposition 4.7]{ASMJM}). For $G\in \Glimm(A)$, we
write $H_G=\iota(G)$. The next result \cite[Proposition 3.2]{AKSI}
is a technical step which is used to move from a
general element of $\Glimm(M(A))$ to an element of the dense subset
$\iota(\Glimm(A))$. We are grateful to the referee for pointing out that the proof in \cite{AKSI} applies to the case $n=1$ as well as to the case $n\geq2$.

\begin{prop}\label{Prop-Glimm-normal} Let $A$ be a $C^*$-algebra with
${\rm Glimm}(A)$ normal. Let $n\geq1$, $H\in {\rm Glimm}(M(A))$ and
$Q_i\in {\rm Prim}(M(A)/H)$ ($1\leq i\leq n$). For $1\leq i\leq n$,
let $N_i$ be an open neighbourhood of $Q_i$ in ${\rm Prim}(M(A))$.
Then there exists $K\in {\rm Glimm}(A)$ and $Q'_i\in {\rm
Prim}(M(A)/H_K)$ such that $Q'_i\in N_i$ ($1\leq i\leq n$).
\end{prop}


\section{Topological properties of $\widehat{G_N}$.}

In this section, we collect some facts concerning the fine structure
of $\widehat{G_N}$ ($N\geq2$). Recall that $G_N=\R^N\rtimes SO(N)$
where $SO(N)$ acts on $\R^N$ by rotation. We embed $SO(N-1)$ into
$SO(N)$ by $SO(N-1)\to {\rm diag}(1, SO(N-1))$. Thus $SO(N-1)$ is
the stability group of characters $\chi_t$ of $\R^N$ corresponding
to vectors $(t, 0, \ldots ,0)\in \R^N$, $t\ne 0$. For $t>0$ and $\
\sigma\in SO(N-1)^{\wedge}\}$, let $\pi_{t,\sigma}= {\rm
ind}^{G_N}_{\R^N\rtimes SO(N-1)}\,\chi_t\times\sigma$, the
irreducible representation of $G_N$ induced by the representation
$\chi_t\times\sigma$ of $\R^N\rtimes SO(N-1)$.
 Then
$$\widehat{G_N}=SO(N)^{\wedge}\cup\{\pi_{t,\sigma}: t>0,\ \sigma\in SO(N-1)^{\wedge}\},$$ where $SO(N)^{\wedge}$ is
considered as a subset of $\widehat{G_N}$ since $G_N/\R^N=SO(N)$.

The topology on $\widehat{G_N}$ is described in \cite{Bag} (see also
\cite{KST}, \cite{EL}, \cite{AEL}). Let
$${\mathcal{U}}_N:=\{\pi_{t,\sigma}: t>0,\
\sigma\in SO(N-1)^{\wedge}\}=\widehat{G_N}\setminus
SO(N)^{\wedge}.$$ The relative topology on $\mathcal{U}_N$ is the
topology induced from the product topology on $(0,\infty)\times
SO(N-1)^{\wedge}$ and of course the relative topology on the closed
subset $SO(N)^{\wedge}$ is discrete. Furthermore, a sequence
$(\pi_{t_n,\sigma_n})_{n\geq1}$ in $\mathcal{U}_N$ is convergent to
some $\pi\in SO(N)^{\wedge}$ if and only if $t_n\to0$ as
$n\to\infty$ and eventually $\sigma_n$ is contained in
$\pi|_{SO(N-1)}$ (see, for example, \cite[Theorem 3.4]{EL}).




If $\pi_{t,\sigma}$ belongs to the Hausdorff space $\mathcal{U}_N$
then $B:=\{\pi_{s,\sigma}:s\geq t/2\}$ is a neighbourhood of
$\pi_{t,\sigma}$ and is closed in $\widehat{G_N}$.  It follows that
$\pi_{t,\sigma}$ is a separated point of $\widehat {G_N}$ and also
that $\ker\pi_{t,\sigma}$ is a Glimm ideal of $C^*(G_N)$ (cf.
\cite[Proposition 7]{Del} and \cite[Proposition 4.9]{KST}). In
summary, $\widehat{G_N}$ consists of the closed subset
$SO(N)^{\wedge}$ which is relatively discrete, together with a dense
open subset of separated points which is the disjoint union of a
countably infinite collection of open half-lines.

\bigskip
Bearing in mind that $SO(2)^{\wedge}=\widehat{\T}=\Z$, we next
recall the representation theory of the groups $SO(N)$ for $N\ge 3$
\cite{Mur} (see also \cite{AEL, EL, Knapp}). Let $k=\lfloor
\frac{N}{2}\rfloor$. The irreducible representations of $SO(N)$ are
parametrized by signatures $(m_1, \ldots ,m_k)\in {\Z}^k$, where
$$m_1\ge m_2\ge \ldots \ge m_{k-1}\ge |m_k| \hbox{ if $N=2k$}$$
$$\hbox{ and } m_1\ge m_2\ge\ldots\ge m_k\ge 0\hbox{ if $N=2k+1$}.$$
Moreover, if $N= 2k$, then
$$(m_1,\ldots , m_k)|_{SO(N-1)}=\sum_{m_i\ge q_i\ge |m_{i+1}|\,\,( 1\le
i\le k-1)} (q_1, \ldots, q_{k-1}),$$
and, if $N=2k+1$, then
$$(m_1,
\ldots, m_k)|_{SO(N-1)}=\sum_{m_i\ge p_i\ge m_{i+1}\,\,( 1\le i\le
k-1),\,\, m_k\ge p_k\ge -m_k} (p_1, \ldots, p_k).$$ Note that in
both cases, the combinatorial condition shows that the number of
summands on the right-hand side is finite. This is, of course,
consistent with the fact that the representation on the left-hand
side is finite dimensional.

\bigskip

Since $\widehat{G_N}\setminus SO(N)^{\wedge}$ consists of separated
points of $\widehat{G_N}$, the next result gives a full description
of the relation $\sim$ of inseparability on $\widehat{G_N}$.

\begin{lemma}\label{motion-sim}  Let $N\geq2$ and
$\pi_1,\pi_2\in SO(N)^{\wedge}\subseteq
\widehat{G_N}$. Then $\pi_1\sim\pi_2$ if and only if
$\pi_1|_{SO(N-1)}$ and $\pi_2|_{SO(N-1)}$ have a common irreducible
subrepresentation.
\end{lemma}

\begin{proof}
Suppose that $\pi_1|_{SO(N-1)}$ and $\pi_2|_{SO(N-1)}$ both contain
some $\sigma\in SO(N-1)^{\wedge}$. Let $t_n=1/n$ ($n\geq1$). Then
$\pi_{t_n,\sigma}\to\pi_1,\pi_2$ as $n\to\infty$ and so
$\pi_1\sim\pi_2$.

Conversely, suppose that $\pi_1\sim\pi_2$. Since $C^*(G_N)$ is
separable and $\mathcal{U}_N$ is dense in $\widehat{G_N}$, it
follows from \cite[Lemma 1.2]{AKSS} that there is a sequence
$(\pi_{t_n,\sigma_n})_{n\geq1}$ in $\mathcal{U}_N$ which is
convergent to both $\pi_1$ and $\pi_2$. Then eventually $\sigma_n$
is contained in both $\pi_1|_{SO(N-1)}$ and $\pi_2|_{SO(N-1)}$.
\end{proof}

\bigskip

 We can now determine ${\rm Orc}(A)$ in the
case where $A=C^*(G_N)$. Since $C^*(G_2)$ is quasi-standard, ${\rm
Orc}(C^*(G_2))=1$.

\begin{prop}\label{G_Nsimpaths} Let $G_N$ be a motion group ($N\geq3$) and suppose that
$k=\lfloor N/2\rfloor$.
\begin{enumerate}
\item[{\rm (i)}] $d((m_1,\ldots , m_k), (n_1, \ldots ,n_k))\le k$ for
$(m_1, \ldots, m_k), (n_1, \ldots , n_k)\in \widehat{SO(N)};$
\item[{\rm (ii)}]
 $d((1, \ldots, 1), (0, \ldots , 0))\ge k$ for
$(0,\ldots,0),(1,\ldots,1)\in SO(N)^{\wedge}$.
\end{enumerate}
\end{prop}

\begin{proof} (i) For $1\le j\le k$, let $s_j=\max\{ m_j, n_j\}$.
Suppose first of all that $N$ is even, so that $N=2k$.
 Since $m_{k-1}\geq|m_k|$,
$$ (m_1, \ldots,m_k)\sim (s_1,m_2, \ldots, m_{k-1},0)$$
because the restrictions to $SO(N-1)$ contain $(m_1,
\ldots,m_{k-1})$. Similarly,
$$ (s_1,m_2, \ldots,m_{k-1},0)\sim (s_1,s_2, m_3\ldots, m_{k-2},0,0)$$
because the restrictions to $SO(N-1)$ contain $(s_1,m_2,
\ldots,m_{k-2},0)$.

Suppose that $k$ is even, say $k=2r$ where $r\geq1$. Then,
proceeding as above, we obtain an $r$-step $\sim$-walk in
$\widehat{G_N}$ from $(m_1,\ldots,m_k)$ to $(s_1, \ldots,s_r,0,
\ldots,0)$ and so $d((m_1,\ldots , m_k), (n_1, \ldots ,n_k))\le 2r=
k$.

Now suppose that $k$ is odd, say $k=2r+1$ where $r\geq1$. In this
case, we obtain an $r$-step $\sim$-walk from $(m_1,\ldots,m_k)$ to
$(s_1, \ldots,s_r,m_{r+1},0, \ldots,0)$ and an $r$-step $\sim$-walk
from $(n_1,\ldots,n_k)$ to $(s_1, \ldots,s_r,n_{r+1},0, \ldots,0)$.
But
$$(s_1, \ldots,s_r,m_{r+1},0, \ldots,0)
\sim(s_1, \ldots,s_r,n_{r+1},0, \ldots,0)$$ because the restrictions
to $SO(N-1)$ contain $(s_1, \ldots,s_r,0, \ldots,0)$. Hence
$$d((m_1,\ldots , m_k), (n_1, \ldots ,n_k))\le 2r+1= k.$$

We now turn to the case where $N$ is odd, so that $N=2k+1$. If
$k=1$, $(m_1)\sim(n_1)$ in $\widehat{G_3}$ because the restrictions
to $SO(2)$ contain the trivial representation $(0)$ of $SO(2)$. So
we now suppose $k\geq 2$. Since $0\geq-m_k$,
$$ (m_1, \ldots,m_k)\sim (s_1,m_2, \ldots, m_{k-1},0)$$
because the restrictions to $SO(N-1)$ contain $(m_1,
\ldots,m_{k-1},0)$. Similarly,
$$ (s_1,m_2, \ldots,m_{k-1},0)\sim (s_1,s_2, m_3\ldots, m_{k-2},0,0)$$
because the restrictions to $SO(N-1)$ contain $(s_1,m_2,
\ldots,m_{k-2},0,0)$. Thus if $k=2r$, we obtain an $r$-step
$\sim$-walk in $\widehat{G_N}$ from $(m_1,\ldots,m_k)$ to $(s_1,
\ldots,s_r,0, \ldots,0)$ and similarly from $(n_1,\ldots,n_k)$ to
$(s_1, \ldots,s_r,0, \ldots,0)$. Hence $d((m_1,\ldots , m_k), (n_1,
\ldots ,n_k))\le 2r= k$.

If $k=2r+1$ (with $r\geq 1$), we obtain an $r$-step $\sim$-walk from
$(m_1,\ldots,m_k)$ to $(s_1, \ldots,s_r,m_{r+1},0, \ldots,0)$ and an
$r$-step $\sim$-walk from $(n_1,\ldots,n_k)$ to $(s_1,
\ldots,s_r,n_{k+1},0, \ldots,0)$. But
$$(s_1, \ldots,s_r,m_{r+1},0, \ldots,0)
\sim(s_1, \ldots,s_r,n_{r+1},0, \ldots,0)$$ because the restrictions
to $SO(N-1)$ contain $(s_1, \ldots,s_r,0, \ldots,0)$. Hence
$$d((m_1,\ldots , m_k), (n_1, \ldots ,n_k))\le 2r+1= k.$$

 (ii) We consider first the case $N=2k$ (so that $k\geq2$). Suppose
that $0\leq i\leq k-2$, $\pi=(M_1,\ldots,M_k)\in SO(N)^{\wedge}$,
$\pi'=(M'_1,\ldots,M'_k)\in SO(N)^{\wedge}$, $\pi\sim\pi'$ in
$\widehat{G_N}$ and $M_j=0$ for $i<j\leq k$. By
Lemma~\ref{motion-sim}, there exists $\sigma=(q_1,\ldots,q_{k-1})\in
SO(N-1)^{\wedge}$ such that
$$M_1\geq q_1\geq M_2\geq \ldots \geq M_i\geq q_i\geq 0 \geq q_{i+1}$$
and
$$M'_1\geq q_1 \geq \ldots\geq M'_{i+1}\geq q_{i+1} \geq \ldots \geq |M'_k|.$$
Thus $q_{i+1}=\ldots=q_{k-1}=0$ and so $M'_{i+2}=\ldots=M'_k=0$.
Applying this with $i=0,1,\ldots,k-2$ in turn, we obtain that
$$d((0,\ldots,0),(1,\ldots,1))\geq k.$$
 Now suppose
that $N=2k+1$. If $k=1$, then $(0)\sim(1)$ in $\widehat{G_3}$
because the restrictions to $SO(2)$ contain the trivial
representation $(0)$ of $SO(2)$. So we now assume $k\geq2$. Suppose
that $0\leq i\leq k-2$, $\pi=(M_1,\ldots,M_k)\in SO(N)^{\wedge}$,
$\pi'=(M'_1,\ldots,M'_k)\in SO(N)^{\wedge}$, $\pi\sim\pi'$ in
$\widehat{G_N}$ and $M_j=0$ for $i<j\leq k$. By
Lemma~\ref{motion-sim}, there exists $\sigma=(p_1,\ldots,p_k)\in
SO(N-1)^{\wedge}$ such that
$$M_1\geq p_1\geq M_2\geq \ldots \geq M_i\geq p_i\geq 0 \geq p_{i+1}\geq\ldots\geq
M_k\geq p_k\geq-M_k$$ and
$$M'_1\geq p_1 \geq \ldots\geq M'_{i+1}\geq p_{i+1} \geq \ldots \geq M'_k
\geq p_k\geq -M_k.$$ Thus $p_{i+1}=\ldots=p_k=0$ and so
$M'_{i+2}=\ldots=M'_k=0$. It follows that
$$d((0,\ldots,0),(1,\ldots,1))\geq k,$$
as required.
\end{proof}

Part (ii) of the next result is contained in \cite[Proposition
4.9]{KST}.

\begin{thm}\label{Orc(G_N)} Let $G_N$ be a motion group with $N\ge 3$ and set
$A=C^*(G_N)$.
\begin{enumerate}
\item[{\rm (i)}] $SO(N)^{\wedge}$ is a $\sim$-connected subset of
$\widehat{G_N}$ and $\Orc(A)=\lfloor N/2\rfloor$.
\item
[{\rm (ii)}] $\Glimm(A)$ consists of the ideals $\ker\pi_{t,\sigma}$
$(t>0,\, \sigma\in SO(N-1)^\wedge)$ and the ideal
$$I_0:=\bigcap\{\ker\pi:\pi\in SO(N)^\wedge\}.$$
\item
[{\rm (iii)}] The complete regularization map
$\phi_A:\Prim(A)\to\Glimm(A)$ is closed.
\end{enumerate}
\end{thm}

\begin{proof} (i) This follows from parts (i) and (ii) of
Proposition~\ref{G_Nsimpaths}.

(ii) Let $\pi\in SO(N)^\wedge$. As noted earlier, each ideal
$\ker\pi_{t,\sigma}$ is a Glimm ideal and, in particular, $
\ker\pi_{t,\sigma} \not\approx \ker\pi$. On the other hand, since
$SO(N)^{\wedge}$ is $\sim$-connected, $\ker\pi' \approx \ker\pi$
for all $\pi'\in SO(N)^\wedge$. Thus $I_0$ is the only other Glimm
ideal.

(iii) Let C be a closed subset of $\Prim(A)$. Then
$\phi_A^{-1}(\phi_A(C))$ is either $C$ or the union of $C$ with the
closed set $\{\ker\pi:\pi\in SO(N)^\wedge\}$, depending on whether
or not $C$ is disjoint from the latter set. Thus $\phi_A$ is closed
with respect to the quotient topology $\tau_q$ on $\Glimm(A)$. Since
$G_N$ is second countable, $A$ is separable and so $\tau_q$
coincides with $\tau_{cr}$ on $\Glimm(A)$ \cite[Theorem 2.6]{Lazar}
(see also \cite[Proposition 4.9]{KST}).
\end{proof}

We remark here that it can be easily seen that the Glimm ideal $I_0$
has no compact neighbourhood in the space $\Glimm(A)$. Thus, even
though $C^*(G_N)$ is separable, $\Glimm(A)$ is neither locally
compact nor first countable \cite[Theorem 3.3]{Lazar}. In
particular, $C^*(G_N)$ is not a $\mathcal{CR}$-algebra in the sense
of \cite{EW}. Moreover, it follows from Theorem~\ref{Orc(G_N)}(ii)
that $A/G$ is non-unital for all $G\in\Glimm(A)$ and so the results
of \cite{ASidealsX} show that there is substantial complexity in the
ideal structure of $M(A)$. In particular there is an injective map
from the lattice of $z$-filters on $\Glimm(A)$ to the lattice of
closed ideals of $M(A)$ \cite[Theorem 3.2]{ASidealsX} and each Glimm
class of $\Prim(M(A))$ which meets the canonical image of $\Prim(A)$
contains at least $2^c$ maximal ideals of $M(A)$ \cite[Theorem
5.3]{ASidealsX}.

\section{\bf  A lower bound for $\Orc(M(A))$.}

 In this section we establish a lower bound for $\Orc(M(A))$,
showing that $\Orc(M(A))\ge \Orc(A)$ under fairly general conditions
(Theorem~\ref{Orc(M(A))geqOrc(A)}). It follows from this that when
$A=C^*(G_N)$, where $G_N$ is a motion group, then
$\Orc(M(A)))\ge\lfloor N/2\rfloor$
(Corollary~\ref{Orc(M(A))motion}).

 For subsets $Y$ and $Z$ of a topological space $X$, let
$$d(Y, Z)=\inf\{ d(y,z): y\in Y,\ z\in Z\}$$
where $d(y,z)$ is as defined at the start of Section 2. For $n\ge
0$, let
$$Y^n=\{ x\in X: d(\{ x\}, Y)\le n\}$$
\noindent and $Y^{\infty}=\{ x \in X: d(\{x\}, Y)<\infty\}$. Note
that $Y^0=Y$, and that $Y^{\infty}=Y^{\Orc(A)}$ if $\Orc(A)<\infty$.

\bigskip

 We shall be interested in C$^*$-algebras $A$ which have the
property that $X^1$ is closed in $\Prim(A)$ whenever $X$ is a closed
subset of $\Prim(A)$. An elementary compactness argument shows that
this property holds when $\Prim(A)$ is compact, see \cite[Corollary
2.3]{Som}, and hence holds whenever $A$ is unital or is the
stabilization of a unital C$^*$-algebra. It also often holds when
$\Prim(A)$ is non-compact.

\bigskip

\begin{lemma}\label{X^1closed} Let $A$ be a C$^*$-algebra and suppose that there is
a closed, relatively discrete subset $Y$ of $\Prim(A)$ such that $Y$
contains all non-singleton $\sim$-components of $\Prim(A)$. Then
$X^1$ is closed whenever $X$ is a closed subset of $\Prim(A)$.
\end{lemma}

\begin{proof} Let $X$ be a closed subset of $\Prim(A)$. Then $X^1= X\cup
(X\cap Y)^1$ since $Y$ contains all the non-singleton
$\sim$-components of $\Prim(A)$. But $(X\cap Y)^1$ is contained in
$Y$ and hence is closed in $\Prim(A)$, because every subset of $Y$
is closed in $Y$ and therefore in $\Prim(A)$. Thus $X^1$ is the
union of two closed sets.
\end{proof}

\bigskip
\noindent When $A=C^*(G_N)$, with $G_N$ a motion group, we may take
$Y$ to be the closed, relatively discrete set $\widehat{SO(N)}$.
Then Lemma~\ref{X^1closed} implies that $X^1$ is closed whenever $X$
is a closed subset of $\Prim(A)$.


\bigskip

If $Y$ and $Z$ are compact subsets of a topological space $X$ such
that $d(Y,Z)\ge 2$ then a routine compactness argument shows that
there exist disjoint open subsets $U$ and $V$ of $X$ with
$Y\subseteq U$ and $Z\subseteq V$ \cite[Lemma 2.2]{Som}. The next
result extends this argument. We say that a topological space is
{\sl locally compact} if every point has a neighbourhood base of
compact sets.

\begin{lemma}\label{Lindelof-lemma} {\bf \cite[Lemma 4.1]{ASomSS}} Let $X$ be a locally compact space and let $Y$ and $Z$ be
subsets of $X$ which are Lindel\"{o}f in the relative topology. Then
the following are equivalent:
\begin{enumerate}
\item
[{\rm (i)}] the closure of $Y^1$ does not meet $Z$ and the closure
of $Z^1$ does not meet $Y$;
\item
[{\rm (ii)}] there exist disjoint open subsets $U$ and $V$ of $X$
with $Y\subseteq U$ and $Z\subseteq V$.
\end{enumerate}
\end{lemma}

 We will apply Lemma~\ref{Lindelof-lemma} to $X=\Prim(A)$, where
$A$ is a $\sigma$-unital C$^*$-algebra. Then $\Prim(A)$ is
$\sigma$-compact, and hence Lindel\"{o}f, so every closed subset of
$X$ is Lindel\"{o}f. Every compact subset of $X$ is Lindel\"{o}f
too, of course.

When $\Prim(A)$ is compact and $\Orc(A)<\infty$, it is automatic
that every Glimm class is $\sim$-connected \cite[Corollary
2.7]{Som}. When $\Prim(A)$ is non-compact, however, there may be
Glimm classes made up of more than one $\sim$-component, even when
$\Orc(A)<\infty$. Furthermore, the complete regularization map
$\phi_A$, which is automatically closed when $\Prim(A)$ is compact,
need not be closed in the general case.

The next result shows that these difficulties do not arise if
$\Orc(A)<\infty$ and $X^1$ is closed for every closed subset $X$ of
$\Prim(A)$.

\begin{prop}\label{phi-closed} Let $A$ be a $\sigma$-unital C$^*$-algebra. Then the following
are equivalent:
\begin{enumerate}
\item[{\rm (i)}] $X^{\infty}$ is closed whenever $X$ is a closed subset
of $\Prim(A)$;
\item[{\rm (ii)}]
 $\phi_A$ is a closed map and every Glimm class is
$\sim$-connected.
\end{enumerate}

\noindent In particular, if $\Orc(A)<\infty$ and $X^1$ is closed for
every closed subset $X$ of $\Prim(A)$ then conditions (i) and (ii)
both hold.
\end{prop}

\begin{proof} (i) $\Rightarrow$ (ii). For $P, Q\in \Prim(A)$ define $P\diamond Q$
if $d_A(P,Q)<\infty$. Then $\diamond$ is an equivalence relation on
$\Prim(A)$ and, for $P,Q\in\Prim(A)$, $P\diamond Q$ implies
$P\approx Q$. Set $W=\Prim(A)/\diamond$ equipped with the quotient
topology, and let $q:\Prim(A)\to W$ be the quotient map. If $X$ is a
closed subset of $\Prim(A)$ then $q^{-1}(q(X))=X^{\infty}$, which is
closed by (i). Hence $q$ is a closed map. If $P\in\Prim(A)$ and
$Q\in\overline{\{P\}}$ then $Q\sim P$. Thus
$\{q(P)\}=q(\overline{\{P\}})$, which is closed in $W$, and so $W$
is a $T_1$-space.

Let $Y$ and $Z$ be non-empty, disjoint closed subsets of $W$. Then
$Y':=q^{-1}(Y)$ and $Z':=q^{-1}(Z)$ are disjoint closed
$\sim$-saturated subsets of $\Prim(A)$. Since $A$ is
$\sigma$-unital, $\Prim(A)$ is $\sigma$-compact and hence
Lindel\"{o}f. Thus $Y'$ and $Z'$ are Lindel\"{o}f, so
Lemma~\ref{Lindelof-lemma} implies the existence of disjoint open
sets $U$ and $V$ containing $Y'$ and $Z'$ respectively.

We now use a standard characterization (see \cite[7.2.14]{Joshi}): a
quotient map $p:X\to D$ is closed if and only if whenever $d\in D$
and $G$ is an open set containing $p^{-1}(d)$ then there exists a
saturated open set $H$ such that $p^{-1}(d)\subseteq H\subseteq G$
(where $H$ is saturated if $H=p^{-1}(p(H))$). Applying this
characterization in the present case to each of the points of $Y$
and $Z$ relative to $U$ and $V$ we obtain $\diamond$-saturated open
sets $U'$ and $V'$ such that $Y'\subseteq U'\subseteq U$ and
$Z'\subseteq V'\subseteq V$. Hence $q(U')$ and $q(V')$ are disjoint
open sets of $W$ containing $Y$ and $Z$ respectively. Hence $W$ is
normal.

Since $W$ is normal and $T_1$, any two distinct points in $W$ can be
separated by a continuous real-valued function. It follows that
$\diamond$ coincides with $\approx$ on $\Prim(A)$ so that
$W=\Glimm(A)$ as sets and $q=\phi_A$. Thus each Glimm class is
$\sim$-connected and $\phi_A$ is closed for the quotient topology on
$\Glimm(A)$. Since we have seen that the quotient topology is normal
and $T_1$, hence completely regular, it coincides with $\tau_{cr}$
(see \cite[Theorem 2.6]{Lazar} for a more general result).


(ii) $\Rightarrow$ (i). (This does not need the $\sigma$-unital
hypothesis). Let $X$ be a closed subset of $\Prim(A)$. Since
$\phi_A$ is a closed and continuous map, $\phi_A^{-1}(\phi_A(X))$ is
closed. But since each Glimm class is $\sim$-connected,
$\phi_A^{-1}(\phi_A(X))=X^{\infty}$.

Finally, suppose that $\Orc(A)<\infty$ and that $X^1$ is closed for
every closed subset $X$ of $\Prim(A)$. Then, for such $X$, the sets
$X^1,(X^1)^1, \ldots,X^{\Orc(A)}(=X^{\infty})$ are all closed
subsets of $\Prim(A)$. \end{proof}

 For a topological space $X$, we now define a \emph{chain of
length} $n$ on $X$ to be a collection of $n$ closed subsets $X_1,
\ldots, X_n$ with the following properties:
\begin{enumerate}
\item[{\rm (i)}] $\bigcup_{i=1}^n X_i=X$;
\item[{\rm (ii)}] $X_i$ and $X_j$ are disjoint if $|i-j|>1$;
\item[{\rm (iii)}] if $n>1$ then the (open) sets $X_1\setminus X_2$ and
$X_n\setminus X_{n-1}$ are non-empty.
\end{enumerate}
\noindent A chain of length $n$ is said to be {\sl admissible} if
there exist $x\in X_1\setminus X_2$ and $y\in X_n\setminus X_{n-1}$
such that $d(x,y)<\infty$. Note that this further condition (in
addition to (i) and (ii)) implies that $X_i\cap X_{i+1}$ is
non-empty for $i=1,\ldots, n-1$, for otherwise $x$ and $y$ would
belong to different clopen subsets of $X$. It was shown in
\cite[Lemma 2.1]{Som} that if $X_1, \ldots, X_n$ is a chain on $X$
of length $n>1$ and $x,y\in X$ with $x\in X_1\setminus X_2$ and
$y\in X_n\setminus X_{n-1}$ then $d(x,y)\ge n$. Hence $\Orc(X)$ is
greater than or equal to the length of any admissible chain on $X$.

\begin{lemma}\label{admissible-chain} Let $A$ be a $\sigma$-unital C$^*$-algebra and suppose that
$X^1$ is closed whenever $X$ is a closed subset of $\Prim(A)$. Let
$\mathcal{S}$ be the class of subsets of $\Prim(A)$ which are either
compact or closed. Suppose that $X, Y\in\mathcal{S}$ with $d(X,Y)\ge
k\ge 2$ and with $X\cup Y$ contained in a single $\sim$-component.
Then there is an admissible chain $X_1, \ldots, X_k$ of closed
subsets of $\Prim(A)$ with $X\subseteq X_1\setminus X_2$ and
$Y\subseteq X_k\setminus X_{k-1}$.
\end{lemma}

\begin{proof} We follow the method of \cite[Lemma 2.4]{Som}. First note that
if $X\in \mathcal{S}$ then $X^1$ is closed, either by assumption if
$X$ is closed, or by \cite[Corollary 2.3]{Som} if $X$ is compact.
Thus the hypotheses imply that the sets $X^n$ and $Y^n$ are closed
for all $n\ge 1$. Since $d(X, Y^{k-2})\ge 2$,  and $X$ and $Y^{k-2}$
are Lindel\"{o}f, it follows from Lemma~\ref{Lindelof-lemma} that
there are disjoint open sets $U_1\supseteq X$ and $V_1\supseteq
Y^{k-2}$. Set $X_1=\Prim(A)\setminus V_1$ and $Y_2=\Prim(A)\setminus
U_1$ and note that $X_1$ and $Y_2$ are closed and that
 $X\subseteq X_1\setminus Y_2$ and $Y^{k-2}\subseteq
Y_2\setminus X_1$. If $k=2$ then $X_1$ and $X_2=Y_2$
have the required properties.

Otherwise, if $k>2$ then $d(X_1, Y^{k-3})\ge 2$ since $X_1$ is
disjoint from $Y^{k-2}$, and $X_1$ and $Y^{k-3}$ are closed and
hence Lindel\"{o}f. For $k>2$ we define inductively, for $i=2,
\ldots, k-1$, $X_i=(\Prim(A)\setminus V_i)\cap Y_i$ and
$Y_{i+1}=\Prim(A)\setminus U_i$, where $U_i$ and $V_i$ are disjoint
open sets containing $X_1\cup\ldots\cup X_{i-1}$ and $Y^{k-(i+1)}$
respectively. Note that for $2\le i\le k-2$, $d((X_1\cup\ldots\cup
X_i), Y^{k-(i+2)})\ge 2$, and $X_1\cup\ldots\cup X_i$ and
$Y^{k-(i+2)}$ are Lindel\"{o}f, so the induction can proceed.
Finally set $X_k=Y_k$. Then $Y\subseteq X_k\setminus X_{k-1}$, and
it is easy to check that $X_1, \ldots, X_k$ is an admissible chain
of length $k$.
\end{proof}

\bigskip

 The next lemma
  asserts a sort
of `normality' for $\Prim(A)$ when $A$ is a $\sigma$-unital
$C^*$-algebra. Recall from Section $2$ that, for $P\in\Prim(A)$, $\tilde P \in \Prim(M(A))$ satisfies $\tilde P \cap A = P$. For $X\subseteq \Prim(A)$, let
$\tilde{X}=\{\tilde{P}:P\in X\}$.

\bigskip

\begin{lemma}\label{disjoint-prim} {\bf \cite[Lemma 2.1]{ASomSS}} Let $A$ be a $\sigma$-unital C$^*$-algebra and let $X$ and $Y$
be disjoint closed subsets of $\Prim(A)$. Then the closures of
$\tilde X$ and $\tilde Y$ are disjoint in $\Prim(M(A))$.
\end{lemma}


\bigskip

\begin{thm}\label{Orc(M(A))geqOrc(A)} Let $A$ be a $\sigma$-unital C$^*$-algebra and suppose that
$X^1$ is closed whenever $X$ is a closed subset of $\Prim(A)$. Then
$\Orc(M(A))\ge \Orc(A)$.
\end{thm}

\begin{proof} We may suppose that $\Orc(A)\geq 2$ for otherwise the result is
trivial. Let $k\in \NN$ with $2\leq k\leq \Orc(A)$. Then there exist
$P, Q\in \Prim(A)$ such that $k\leq d_A(P,Q)<\infty$.
By Lemma~\ref{admissible-chain}, applied to the compact sets $\{P\}$
and $\{Q\}$, there is an admissible chain $X_1, \ldots X_k$ on
$\Prim(A)$, of length $k$. For each $1\le i\le k$, let $Y_i$ be the
closure of $\tilde X_i=\{\tilde R: R\in X_i\}$ in $\Prim(M(A))$. It
then follows, using Lemma~\ref{disjoint-prim} to check the
preservation of disjointness, that $Y_1,\ldots, Y_k$ is a chain on
$\Prim(M(A))$ of length $k$. Since $d_{M(A)}(\tilde P, \tilde Q)\le
d_A(P,Q) < \infty$, this chain is admissible. Hence $\Orc(M(A))\ge
k$ by \cite[Lemma 2.1]{Som}.
\end{proof}

\begin{cor}\label{Orc(M(A))motion} Let $G_N$ be a motion group with $N\ge 3$ and set
$A=C^*(G_N)$. Then $\Orc(M(A))\ge \lfloor N/2\rfloor$.
\end{cor}

\begin{proof}This follows from Theorem~\ref{Orc(G_N)} and
Theorem~\ref{Orc(M(A))geqOrc(A)}, and from Lemma~\ref{X^1closed} and
the remark following.
\end{proof}

\bigskip

\section{Upper bounds for $K(M(A))$.}

In this section, motivated by the motion groups $G_N$, we obtain
some general $C^*$-theoretic results for the constants $K(A)$ and
$K(M(A))$.

\bigskip

Let ${\rm Id}(A)$ be the set of all closed two-sided ideals of a
$C^*$-algebra $A$. This is a compact Hausdorff space for the
topology defined by Fell in \cite[Section II]{Fell}. We denote this
topology by $\tau_s$ and we recall that a net $(J_{\alpha})$ is
$\tau_s$-convergent to $J$ in ${\rm Id}(A)$ if and only if $\Vert
a+J_{\alpha}\Vert\to\Vert a+J\Vert$ for all $a\in A$ (see
\cite[Theorem 2.2]{Fell}). A (closed two-sided) ideal $J$ of $A$ is
said to be \emph{primal} if whenever $n\geq2$ and
$J_1,J_2,\ldots,J_n$ are ideals of $A$ with product $J_1J_2\ldots
J_n=\{0\}$ then at least one of the $J_i$ is contained in $J$. This
concept arose in \cite{AB} where it was shown that a state of $A$ is
a weak$^*$-limit of factorial states if and only if the kernel of
the associated GNS-representation is primal. It follows from
\cite[Proposition 3.2]{AB} that an ideal $J$ of $A$ is primal if and
only if there exists a net in $\Prim(A)$ which is convergent to
every point in some dense subset of $\Prim(A/J)$. The set of all
primal ideals of $A$ is $\tau_s$-closed in $\rm{Id}(A)$ (see
\cite[p.531]{Rob}).  $\Primal'(A)$ is the set of proper primal
ideals of $A$, Min-Primal($A$) is the set of minimal primal ideals
of $A$ and $\Sub(A)$ is the $\tau_s$-closure of $\Min$-$\Primal(A)$
in $\rm{Id}(A)\setminus\{A\}$ and is therefore contained in
$\Primal'(A)$ \cite[p.84]{MScand}.

When $A$ is separable, $(\Primal'(A), \tau_s)$ is metrizable
\cite[Lemme 2]{DixCJM} and the set of separated points of $\Prim(A)$
is $\tau_s$-dense in $\Min$-$\Primal(A)$ \cite[Corollary 4.6]{Rob},
so for $I\in \Sub(A)$ there is a sequence $(P_n)$ of separated
points of $\Prim(A)$ such $P_n\to I$ $(\tau_s)$. Hence $\Prim(A/I)$
is precisely the set of limits of $(P_n)$ in $\Prim(A)$ and every
cluster point of $(P_n)$ is a limit (see \cite[Theorem 2.1]{Fe} and
\cite[Lemma 1.4]{MScand}). Conversely suppose that $(P_n)$ is a
convergent sequence of separated points in $\Prim(A)$ and that every
cluster point of $(P_n)$ is a limit. Let $X$ be the set of limits of
$(P_n)$. Then $P_n\to I=\ker X$ $(\tau_s)$ and so $I\in \Sub(A)$.
Thus, given a separable $C^*$-algebra $A$ and a description of
$\Prim(A)$ as a topological space, it is usually possible to
identify $\Sub(A)$. If $A$ is quasi-standard then
$\Glimm(A)=\Min$-$\Primal(A)= \Sub(A)$ as sets and topological
spaces \cite[Theorem 3.3]{ASo}. On the other hand, if
$\Glimm(A)=\Min$-$\Primal(A)$ (as sets) and if $A$ is not
quasi-standard then $\Sub(A)$ strictly contains $\Min$-$\Primal(A)$
\cite[Theorem 3.3 ((v)$\rightarrow$(i))]{ASo}. This phenomenon
occurs for $C^*(SL(2,\C))$ \cite[Example 4.1]{AKSI}.
We shall determine $\Sub(C^*(G_N))$ in
Proposition~\ref{motion-sub}.

\medskip

The next result was obtained in \cite[Theorem 5.2]{AKSI}. It will
later be applied to the cases $N\equiv0$ and $N\equiv3$ (${\rm
mod}\,4$).

\begin{thm}\label{M(A)and(n+1)} Let $A$ be a C$^*$-algebra with $\Glimm(A)$ normal
and
$\phi_A$ closed. Suppose that there exists $n\ge 0$ such that
whenever $G\in \Glimm(A)$ and $I^{(i)}\in {\rm Sub}(A)$ $(1\le i\le
3)$ with $I^{(1)}\cap I^{(2)}\cap I^{(3)}\supseteq G$ then there
exist $S^{(i)}\in \Prim(A)$ $(1\le i\le 3)$ with $I^{(i)}\subseteq
S^{(i)}$ and $T\in \Prim(A)$ with $d_A(S^{(i)}, T)\le n$ $(1\le i\le
3)$. Then $K(M(A))\le n+1$.
\end{thm}

 Our second general result on $K(M(A))$, Theorem~\ref{M(A)and(n+3/2)}
below, will subsequently be applied to the remaining cases
$N\equiv1$ and $N\equiv2$ (${\rm mod}\,4$). But first of all, we
need a new result for unital $C^*$-algebras.

\begin{thm}\label{unital(n+1/2)} Let $A$ be a unital C$^*$-algebra and let $n$
be a positive integer.
 Suppose that whenever $P, Q, R$ are primitive ideals of $A$ lying
in the same Glimm class there exist $S, T, U\in \Prim(A)$ such that
$d_A(P,S)\le n$, $d_A(Q, T)\le n$, $d_A(R, U)\le n$, and $S\cap
T\cap U$ is primal. Then $K(A)\le n+1/2$.
\end{thm}

\begin{proof} Let $a\in A$ with $\Vert D(a,A)\Vert\le 1$. We want to
show that $d(a,Z(A))\le n+1/2$. Let $K\in \Glimm(A)$. Then by
\cite[Theorem 2.3]{Some} it suffices to show that $\Vert
a_K-\lambda(a_K)\Vert\le n+1/2$, where $a_K$ is the canonical image of $a$ in $A/K$ and $\lambda(a_K)$ is the
unique scalar multiple of the identity  in $A/K$ which is closest to $a_K$. Set $b=a-\lambda (a_K)$. Then
$\Vert D(b,A)\Vert\le 1$ and $\lambda(b_K)=0$. We show that $\Vert
b_K\Vert\le n+1/2$.

Write $r=\Vert b_K\Vert$. Let $C$ be the circle of radius $r$
centred at the origin. Then there exist $x, y, z\in C\cap U(b_K,
A/K)$ and extreme points $f', g', h'$ of $N(A/K)$ such that $C$ is
the bounding circle of $\{ x, y, z \}$ and $f'(b_K)=x$, $g'(b_K)=y$,
$h'(b_K)=z$. Let $\epsilon>0$ be given. By Milman's Theorem there
exist $f, g, h\in G(A/K)\subseteq G(A)$ with $f(1), g(1), h(1)>0$
such that $|f'(b)-f(b)|<\epsilon$, $|g'(b)-g(b)|<\epsilon$, and
$|h'(b)-h(b)|<\epsilon$. Set $P=\Gamma(f)$, $Q=\Gamma(g)$, and
$R=\Gamma(h)$.

Since $f, g, h\in G(A/K)$, it follows that $P, Q, R\supseteq K$.
Hence there exist $S, T, U\in \Prim(A)$ such that $d_A(P, S)\le n$,
$d_A(Q, T)\le n$, $d_A(R, U)\le n$, and $S\cap T\cap U$ primal. Let
$V=S\cap T\cap U$.

Since $d_A(P,S)\le n$, there exist $P_1,\ldots P_n\in \Prim(A)$ such
that $P\sim P_1\sim\ldots\sim P_n=S$. By \cite[Proposition 1.3]{Id},
$$|\lambda(b_{P_n})-\lambda(b_V)|\le
1/2,\eqno{(1)}$$ $$|\lambda (b_{P\cap P_1})-\lambda(b_{P_1})|\le
1/2, \eqno{(2)}$$ and, for $1\le i\le n-1$,
$$|\lambda(b_{P_i})-\lambda(b_{P_{i+1}})|\le
|\lambda(b_{P_i})-\lambda(b_{P_i\cap P_{i+1}})|+|\lambda (b_{P_i\cap
P_{i+1}}) -\lambda (b_{P_{i+1}})|\le 1.\eqno{(3)}$$ Writing
$\mu=\lambda(b_V)$, we obtain from $(1)$, $(2)$, and $(3)$ that
$$|\lambda(b_{P\cap P_1})-\mu|\le 1/2+(n-1)+1/2=n. \eqno{(4)}$$
On the other hand, since $P\cap P_1$ is primal, it follows from
\cite[Proposition 2.6]{Some} that
$$\Vert b_{P\cap P_1}-\lambda (b_{P\cap P_1})\Vert\le 1/2.
\eqno{(5)}$$ Since $f$ factors through $A/P$ and hence through
$A/(P\cap P_1)$, it follows from $(4)$ and $(5)$ that
$$|x-f(1)\mu|\le |x-f(b)|+|f(b_{P\cap P_1})-\lambda(b_{P\cap P_1})f(1)|+f(1)|\lambda(b_{P\cap P_1})
-\mu|$$
$$<\epsilon+\Vert f\Vert/2+f(1)n\le n+1/2+\epsilon.$$ Similarly $|y-g(1)\mu|\le n+1/2+\epsilon$ and
$|z-h(1)\mu|\le n+1/2+\epsilon$.

If $\mu=0$ then we have that $\Vert b_K\Vert=|x|\le n+1/2+\epsilon$.
If $\mu\ne 0$ then we produce the line from $\mu$ to $0$ to meet the
circle $C$ at a point $E$, say. Let $FH$ be the diameter of $C$
perpendicular to $OE$. Then the semicircle $FEH$ meets $\{x, y, z
\}$ in $x$, say. Hence $\Vert b_K\Vert\le |x-f(1)\mu|\le n+1/2+
\epsilon$. So in either case $\Vert b_K\Vert\le n+1/2+\epsilon$.
Since $\epsilon$ was arbitrary, $\Vert b_K\Vert\le n+1/2$ as
required.
\end{proof}

\begin{thm}\label{M(A)and(n+3/2)} Let $A$ be a C$^*$-algebra such that ${\rm Glimm}(A)$ is
normal and $\phi_A$ is closed. Suppose that there exists $n\ge 0$
such that whenever $G\in \Glimm(A)$ and $I^{(i)}\in \Sub(A)$ $(1\le
i\le 3)$ with $I^{(1)}\cap I^{(2)}\cap I^{(3)}\supseteq G$ then
there exist $S^{(i)}, T^{(i)}\in \Prim(A)$ $(1\le i\le 3)$ with
$I^{(i)}\subseteq S^{(i)}$, $d_A(S^{(i)}, T^{(i)})\le n$ $(1\le i\le
3)$, and $T^{(1)}\cap T^{(2)}\cap T^{(3)}$ primal. Then $K(M(A))\le
n+3/2$.
\end{thm}

\begin{proof} We show that $M(A)$ satisfies the hypotheses of Theorem~\ref{unital(n+1/2)}
 with $n$ replaced by $n+1$. Suppose that $H\in {\rm Glimm}(M(A))$
and $Q^{(i)}\in {\rm Prim}(M(A)/H)$ ($1\leq i\leq3$). Let
$\mathcal{L}$ (respectively $\mathcal{M}$, $\mathcal{N}$) be a base
of open neighbourhoods of $Q^{(1)}$ (respectively $Q^{(2)}$,
$Q^{(3)}$) in $\Prim(M(A))$. Let
$\Delta=\mathcal{L}\times\mathcal{M}\times \mathcal{N}$ with the
usual order.

Temporarily fix $\alpha= (L, M, N)\in \Delta$. Exactly as in the
proof of Theorem~\ref{M(A)and(n+1)} (see \cite[Theorem 5.2]{AKSI}),
we may apply Proposition~\ref{Prop-Glimm-normal} to obtain $K_{\alpha}\in\Glimm(A)$ and $I^{(i)}_{\alpha}\in\Sub(A)$
($1\leq i\leq3$) such that $K_{\alpha}\subseteq I^{(1)}_{\alpha}\cap
I^{(2)}_{\alpha}\cap I^{(3)}_{\alpha}$.
 By hypothesis there exist $T^{(i)}_{\alpha,j}\in {\rm Prim}(A)$
($1\leq i\leq 3,\,1\leq j\leq n+1$) such that
$$I^{(i)}_{\alpha}\subseteq T^{(i)}_{\alpha,1}\sim\ldots\sim
T^{(i)}_{\alpha,n+1} \qquad (1\leq i\leq3)\eqno{(1)}$$ and
$\bigcap^3_{i=1}T^{(i)}_{\alpha,n+1}  $ is a primal ideal of $A$.

We now let $\alpha$ vary. By the compactness of ${\rm Prim}(M(A))$
and by passing to successive subnets, we obtain $T^{(i)}_j\in {\rm
Prim}(M(A))$ and commonly indexed subnets
$(T^{(i)}_{\alpha(\beta),j})$ in ${\rm Prim}(A)$ such that
$$ \tilde{T}^{(i)}_{\alpha(\beta),j}\to T^{(i)}_j \qquad (1\leq
i\leq3,\,1\leq j\leq n+1).$$ It follows from (1) that
$$ T^{(i)}_{1}\sim\ldots\sim
T^{(i)}_{n+1} \qquad (1\leq i\leq3).$$

We show next that $\bigcap^3_{i=1}T^{(i)}_{n+1}$ is a primal ideal
of $M(A)$. Let $V_i$ be an open neighbourhood of $T^{(i)}_{n+1}$ in
${\rm Prim}(M(A))$ ($1\leq i\leq3$). There exists $\beta$ such that
$$ \tilde{T}^{(i)}_{\alpha(\beta),n+1} \in V_i \qquad (1\leq i\leq3).$$
Since $\bigcap^3_{i=1}T^{(i)}_{\alpha(\beta),n+1}  $ is primal,
there is a net in ${\rm Prim}(A)$ convergent to all of
$T^{(1)}_{\alpha(\beta),n+1}$, $T^{(2)}_{\alpha(\beta),n+1}$ and
$T^{(3)}_{\alpha(\beta),n+1}$, and hence a net in ${\rm Prim}(M(A))$
convergent to all of $\tilde{T}^{(i)}_{\alpha(\beta),n+1}$ ($1\leq
i\leq3$). Hence $V_1\cap V_2\cap V_3$ is non-empty as required.

Exactly as in the proof of Theorem~\ref{M(A)and(n+1)} (see
\cite[Theorem 5.2]{AKSI} again), we have $Q^{(i)}\sim T^{(i)}_1$
($1\leq i\leq3$). Thus we have shown that
$d_{M(A)}(Q^{(i)},T^{(i)}_{n+1})\leq n+1$ ($1\leq i\leq 3$) and that
$\bigcap^3_{i=1}T^{(i)}_{n+1}$ is a primal ideal of $M(A)$. It
follows from Theorem~\ref{unital(n+1/2)} that $K(M(A))\leq (n+1)
+\frac{1}{2} = n+\frac{3}{2}$.
\end{proof}

\section{The cases $N\equiv0$ and $N\equiv2 \mod 4$.}

In order to apply the results of the previous section, we need to
begin by determining $\Sub(C^*(G_N))$.
\bigskip

 For
$\sigma\in SO(N-1)^{\wedge}$, let
$$I_{0,\sigma}=\bigcap\{\ker\pi:\pi\in
SO(N)^{\wedge},\,\pi|_{SO(N-1)}\geq\sigma\}.$$ Since the closed
subset $SO(N)^{\wedge}$ of $\widehat{G_N}$ is relatively discrete,
the set $$\{\ker\pi:\pi\in
SO(N)^{\wedge},\,\pi|_{SO(N-1)}\geq\sigma\}$$ is a closed subset of
${\rm Prim}(C^*(G_N))$ and therefore is the hull of the ideal
$I_{0,\sigma}$. Let $(t_n)$ be any null sequence in $(0,\infty)$.
Then, for all $P$ in the hull of $I_{0,\sigma}$,
$\ker\pi_{t_n,\sigma}\to P$ as $n\to\infty$. On the other hand,
suppose that $Q\in\Prim(C^*(G_N))$ is a cluster point of
$(\ker\pi_{t_n,\sigma})_{n\geq1}$. Since $C^*(G_N)$ is separable,
$Q$ has a countable base of neighbourhoods in $\Prim(C^*(G_N))$ and
so there is a subsequence $(\ker\pi_{t_{n_k},\sigma})_{k\geq1}$
convergent to $Q$. Since $t_{n_k}\to0$ as $k\to\infty$, $Q=\ker\pi$
for some $\pi\in SO(N)^{\wedge}$ such that $\sigma$ is contained in
$\pi|_{SO(N-1)}$, and hence $Q\supseteq I_{0,\sigma}$. It now
follows that $\ker\pi_{t_n,\sigma}\to_{\tau_s}I_{0,\sigma}$ as
$n\to\infty$. Since each $\ker\pi_{t_n,\sigma}$ is a separated point
of ${\rm Prim}(C^*(G_N))$, and hence a minimal primal ideal of
$C^*(G_N)$ \cite[Proposition 4.5]{Rob}, we obtain that
$I_{0,\sigma}\in{\rm Sub}(C^*(G_N))$.

\begin{prop}\label{motion-sub} Let $A=C^*(G_N)$ ($N\geq2$). Then
$${\rm Sub}(A) = \{I_{0,\sigma}:\sigma\in SO(N-1)^{\wedge}\}\cup
\{\ker\pi_{t,\sigma}:t>0,\sigma\in SO(N-1)^{\wedge}\}.$$ In
particular, $\Sub(A)=\Min$-$\Primal(A)$ if and only if $N$ is even.
\end{prop}

\begin{proof} Suppose that $I\in{\rm Sub}(A)$ and consider $\pi\in\widehat{A}$
such that $\ker\pi\supseteq I$. If $\pi\in\mathcal{U}_N$ then the
minimal primal ideal $\ker\pi$ contains the primal ideal $I$ and so
$I=\ker\pi$. So we may assume from now on that ${\rm
hull}(I)\subseteq\{\ker\pi:\pi\in SO(N)^{\wedge}\}$. By the
discussion at the start of Section $5$, there is a sequence $(P_n)$
of separated points of ${\rm Prim}(A)$ such that $P_n\to_{\tau_s}I$
as $n\to\infty$. In particular, for $P\in{\rm Prim}(A)$, $P_n\to P$
as $n\to\infty$ if and only if $P\supseteq I$.

For each $n\geq1$ there exists $t_n>0$ and $\sigma_n\in
SO(N-1)^{\wedge}$ such that $P_n= \ker\pi_{t_n,\sigma_n}$. Let
$\pi\in SO(N)^{\wedge}$ such that $\ker\pi\supseteq I$. Then
$\pi_{t_n,\sigma_n}\to\pi$ as $n\to\infty$ and so $(t_n)$ is a null
sequence and eventually $\sigma_n$ is contained in $\pi|_{SO(N-1)}$.
Replacing $(P_n)$ by a subsequence (which will also be
$\tau_s$-convergent to $I$), we may assume that $\sigma_n=\sigma$
(say) for all $n$. Then
$P_n=\ker\pi_{t_n,\sigma}\to_{\tau_s}I_{0,\sigma}$ as $n\to\infty$,
as observed at the start of this section. Since $\tau_s$ is
Hausdorff, $I=I_{0,\sigma}$ as required.

The final statement of the proposition follows from the
characterization of $\Min$-$\Primal(A)$ in \cite[Proposition
4.6]{KST}.
\end{proof}

The next two results deal with the cases $N\equiv0$ and $N\equiv2
\mod 4$. The strategies are similar in that both proofs use
Corollary~\ref{Orc(M(A))motion} for one estimate. However, for the
other estimate, the first case uses Theorem~\ref{M(A)and(n+1)}
whereas the second case requires the more complicated
Theorem~\ref{M(A)and(n+3/2)}. For the application of Theorems
\ref{M(A)and(n+1)} and ~\ref{M(A)and(n+3/2)}, we recall that, for
the separable $C^*$-algebra $A=C^*(G_N)$ ($N\geq3$), $\Glimm(A)$ is
normal (see Section 2) and $\phi_A$ is closed
(Theorem~\ref{Orc(G_N)}).

 \begin{thm}\label{Ncong0} Let $A=C^*(G_N)$ where
$N\equiv 0$ $({\rm mod}\ 4)$. Then
$$K(M(A))=K_s(M(A))=\frac{1}{2}{\rm Orc}(M(A))=N/4.$$
\end{thm}

\begin{proof} We begin with the case $N>4$ so that $N/4>1$. We aim to show
that $A$ satisfies the hypotheses of Theorem~\ref{M(A)and(n+1)} with
$n=\frac{N}{4}-1$. Let $I^{(1)},I^{(2)},I^{(3)}\in \Sub(A)$ with
$I^{(1)},I^{(2)},I^{(3)}$ all containing the same Glimm ideal $G$.
If $G$ is one of the separated points of $\Prim(A)$ then
$I^{(1)}=I^{(2)}=I^{(3)}=G$ and we may take
$S^{(1)}=S^{(2)}=S^{(3)}=T=G$.

We may suppose therefore that $G=I_0$ (see Theorem~\ref{Orc(G_N)})
and that $I^{(1)}=I_{0, \sigma}$, $I^{(2)}=I_{0, \sigma'}$, and
$I^{(3)}=I_{0, \sigma''}$ for some $\sigma, \sigma', \sigma''\in
SO(N-1)^{\wedge}$ by Proposition~\ref{motion-sub}. Let $\sigma=(q_1,
\ldots, q_{k-1})$, $\sigma'=(q'_1,\ldots, q'_{k-1})$, and
$\sigma''=(q''_1,\ldots, q''_{k-1})$, where $k=N/2$. Set $Q_i=\max\{
q_i, q'_i, q''_i\}$ for $1\le i\le N/4$. Define $\pi= (Q_1, q_1,
q_2, \ldots, q_{k-2},0)\in SO(N)^{\wedge}$ and similarly $\pi'$ and
$\pi''$ (replacing the $q_j$ by $q_j'$ and $q_j''$ respectively).
Since $\pi|_{SO(N-1)} \geq \sigma$, $\ker\pi\supseteq I^{(1)}$.
Similarly, $\ker\pi'\supseteq I^{(2)}$ and $\ker\pi''\supseteq
I^{(3)}$.

There is an $(\frac{N}{4}-1)$-step $\sim$-walk from each of $\pi$,
$\pi'$, $\pi''$ to $\rho:=(Q_1,\ldots , Q_{\frac{N}{4}}, 0,\ldots ,
0)$ in $\widehat{G_N}$. To see this, note that for $1\le i\le
\frac{N}{4}-2$,
$$(Q_1, \ldots , Q_i, q_i, q_{i+1}, \ldots , q_{k-1-i}, 0, \ldots, 0)\sim
(Q_1, \ldots, Q_i, Q_{i+1}, q_{i+1}, \ldots, q_{k-2-i}, 0, \ldots ,
0)$$ because the restrictions to $SO(N-1)$ contain $(Q_1,\ldots,
Q_i, q_{i+1}, \ldots, q_{k-1-i}, 0,\ldots, 0),$ and finally
$$(Q_1, \ldots, Q_{\frac{N}{4}-1}, q_{\frac{N}{4}-1}, q_{\frac{N}{4}}, 0, \ldots, 0) \sim
(Q_1, \ldots , Q_{\frac{N}{4}}, 0, \ldots, 0)=\rho$$ because the
restrictions to $SO(N-1)$ contain $(Q_1, \ldots, Q_{\frac{N}{4}-1},
q_{\frac{N}{4}}, 0, \ldots, 0)$. Similar arguments apply to $\pi'$
and $\pi''$, replacing the $q_j$ by $q_j'$ and $q_j''$
(respectively).

Thus taking $S^{(1)}=\ker\pi$, $S^{(2)}=\ker\pi'$,
$S^{(3)}=\ker\pi''$, and $T=\ker\rho$ we have satisfied the
hypotheses of Theorem~\ref{M(A)and(n+1)}. Thus $K(M(A))\le
\frac{N}{4}$. Combining this with Theorem~\ref{Orc(G_N)}(i),
Corollary~\ref{Orc(M(A))motion} and \cite[Theorem 4.4]{Som}, we have
$$K(M(A))\le \frac{N}{4}\leq \frac{{\rm Orc}(A)}{2}
\leq \frac{{\rm Orc}(M(A))}{2}=K_s(M(A))\leq K(M(A))$$
 and hence equality throughout.

 For the simpler case $N=4$, it again suffices to check the hypothesis of
 Theorem~\ref{M(A)and(n+1)} for $G=I_0$ (the only non-maximal Glimm ideal of $A$).
 So let
 $I^{(1)}=I_{0,(q)}$, $I^{(2)}=I_{0,(q')}$ and $I^{(3)}=I_{0,(q'')}$ where
 $(q),(q'),(q'')\in SO(3)^{\wedge}$. Let $Q=\max\{q,q',q''\}$ and
 let $\pi=(Q,0)\in SO(4)^{\wedge}$. Then $\pi|_{SO(3)}$ contains
 $(q)$,$(q')$ and $(q'')$ and so $\ker\pi\supseteq I^{(1)}+I^{(2)}+I^{(3)}$. Using
 Theorem~\ref{M(A)and(n+1)} in the case $n=0$ and arguing as above, we have
 $$K(M(A))\le 1\leq \frac{{\rm Orc}(A)}{2}
\leq \frac{{\rm Orc}(M(A))}{2}=K_s(M(A))\leq K(M(A))$$
 and hence equality throughout.
 \end{proof}

\begin{thm}\label{Ncong2} Let $A=C^*(G_N)$ where $N\equiv 2$ $({\rm mod}\ 4)$ and
$N\geq6$. Then
$$K(M(A))=K_s(M(A))=\frac{1}{2}{\rm Orc}(M(A))=N/4.$$
\end{thm}

\begin{proof} Let $k=N/2$ and $m=(N-2)/4$, so that $k=2m+1$. We begin with
the case $N>6$ so that $m>1$. We aim to show that $A$ satisfies the
hypotheses of Theorem~\ref{M(A)and(n+3/2)} with $n=m-1$. Let
$I^{(1)},I^{(2)},I^{(3)}\in \Sub(A)$ with $I^{(1)},I^{(2)},I^{(3)}$
all containing the same Glimm ideal $G$. If $G$ is one of the
separated points of $\Prim(A)$ then $I^{(1)}=I^{(2)}=I^{(3)}=G$ and
we may take $S^{(i)}=T^{(i)}=G$ $(1\leq i\leq3)$ so that
$T^{(1)}\cap T^{(2)}\cap T^{(3)}=G$ which is primal.

We may suppose therefore that $G=I_0$ (see Theorem~\ref{Orc(G_N)})
and that $I^{(1)}=I_{0, \sigma}$, $I^{(2)}=I_{0, \sigma'}$, and
$I^{(3)}=I_{0, \sigma''}$ for some $\sigma, \sigma', \sigma''\in
SO(N-1)^{\wedge}$ by Proposition~\ref{motion-sub}. Let $\sigma=(q_1,
\ldots, q_{k-1})$, $\sigma'=(q'_1,\ldots, q'_{k-1})$, and
$\sigma''=(q''_1,\ldots, q''_{k-1})$. Set $Q_i=\max\{ q_i, q'_i,
q''_i\}$ for $1\le i\le m$. Define $\pi= (Q_1, q_1, q_2, \ldots,
q_{k-2},0)\in SO(N)^{\wedge}$ and similarly $\pi'$ and $\pi''$
(replacing the $q_j$ by $q_j'$ and $q_j''$ respectively). Since
$\pi|_{SO(N-1)} \geq \sigma$, $\ker\pi\supseteq I^{(1)}$. Similarly,
$\ker\pi'\supseteq I^{(2)}$ and $\ker\pi''\supseteq I^{(3)}$.

There is an $(m-1)$-step $\sim$-walk in $\widehat{G_N}$ from $\pi$
to $\rho:=(Q_1,\ldots , Q_m,q_{m+1} 0,\ldots , 0)\in
SO(N)^{\wedge}$. To see this (as in the proof of
Theorem~\ref{Ncong0}), note that for $1\le i\le m-1$,
$$(Q_1, \ldots , Q_i, q_i, q_{i+1}, \ldots , q_{k-1-i}, 0, \ldots, 0)\sim
(Q_1, \ldots, Q_i, Q_{i+1}, q_{i+1}, \ldots, q_{k-2-i}, 0, \ldots ,
0)$$ because the restrictions to $SO(N-1)$ contain $(Q_1,\ldots,
Q_i, q_{i+1}, \ldots, q_{k-1-i}, 0,\ldots, 0),$ Similarly, there is
an $(m-1)$-step $\sim$-walk from $\pi'$ to $\rho':=(Q_1,\ldots ,
Q_m,q'_{m+1}, 0,\ldots , 0)$ and from $\pi''$ to
$\rho'':=(Q_1,\ldots , Q_m,q''_{m+1}, 0,\ldots , 0)$.

  The restrictions of $\rho$, $\rho'$ and $\rho''$ to
  $SO(N-1)$ contain $\mu:=(Q_1,\ldots,Q_m,0,\ldots,0)$ as a
  common subrepresentation and so if $(t_n)$ is any null sequence in $(0,\infty)$ then
  $\pi_{t_n,\mu}\to \rho,\rho',\rho''$ in $\widehat{G_N}$
  as $n\to\infty$. Hence
$\ker\rho\cap\ker\rho'\cap\ker\rho''$ is a primal ideal of
  $A$.

Taking $S^{(1)}=\ker\pi$, $S^{(2)}=\ker\pi'$, $S^{(3)}\ker\pi''$ and
 $T^{(1)}=\ker\rho$, $T^{(2)}=\ker\rho'$, $T^{(3)}\ker\rho''$, we have satisfied
 the hypotheses of Theorem~\ref{M(A)and(n+3/2)}.
Thus $K(M(A))\le (m-1)+3/2=N/4$. Combining this with
Theorem~\ref{Orc(G_N)}(i), Corollary~\ref{Orc(M(A))motion} and
\cite[Theorem 4.4]{Som}, we have
$$K(M(A))\le \frac{N}{4}\leq \frac{{\rm \Orc}(A)}{2}
\leq \frac{{\rm \Orc}(M(A))}{2}=K_s(M(A))\leq K(M(A))$$
 and hence equality throughout.

 For the simpler case $N=6$, it again suffices to check the hypothesis of
 Theorem~\ref{M(A)and(n+3/2)}
 for the Glimm ideal  $G=I_0$. With notation as above, we have that
 the restrictions to $SO(N-1)$ of $\pi=(Q_1,q_1,0)$,
 $\pi'=(Q_1,q'_1,0)$ and $\pi''=(Q_1,q''_1,0)$ contain $(Q_1,0)$ as a
 common subrepresentation and so the ideal $\ker\pi\cap\ker\pi'\cap\ker\pi''$
 is primal. Applying Theorem~\ref{M(A)and(n+3/2)} in the case $n=0$, we have
$$K(M(A))\le \frac{3}{2}\leq \frac{{\rm Orc}(A)}{2}
\leq \frac{{\rm Orc}(M(A))}{2}=K_s(M(A))\leq K(M(A))$$
 and hence equality throughout.
 \end{proof}

\section{The constant $D(A)$ and the value of $\Orc(M(A))$.}

If we apply the strategy of the previous section to the case where
$N$ is odd, we find that $K(M(C^*(G_N)))\in
[\frac{N-1}{4},\frac{N+1}{4}]$ rather than obtaining an exact value.
This forces us to improve on the estimate for $\Orc(M(C^*(G_N)))$
that was given in Corollary~\ref{Orc(M(A))motion}.

In the first part of this section we obtain an upper bound for
$\Orc(M(A))$, in terms of the ideal structure of $A$, which is
applicable to a fairly general class of $C^*$-algebras $A$. In the
second part, we obtain the precise value of $\Orc(M(A))$ for a
smaller class of $C^*$-algebras which does, however, contain the
group $C^*$-algebras of the motion groups.


Recall from Section $5$ that, for a C$^*$-algebra $A$, $\Sub(A)$ is
the $\tau_s$-closure of $\Min$-$\Primal(A)$ in $\Primal'(A)$. We now
define a graph structure on $\Sub(A)$. For $I, J\in \Sub(A)$ write
$I*J$ if $I+J\ne A$. The relation $*$ defines a graph structure on
$\Sub(A)$ analogous to the graph structure on $\Prim(A)$ defined by
the relation $\sim$. Let $d^*(I,J)$ denote the distance between $I$
and $J$ in the graph $(\Sub(A), *)$. As before, we define the
diameter of a $*$-component to be the supremum of the distances
between points in the component, with the exception that this time
we define the diameter of a singleton to be $0$ (rather than $1$ as
in $(\Prim(A),\sim)$). Let $D(A)$ be the supremum of the diameters
of $*$-connected components of $\Sub(A)$.

For example, let $A=C^*(SL(2,{\bf C}))$ (see \cite[Example
5.1]{AKSI} for notation). Then, apart from loops, the only edge in
$(\Sub(A),*)$ is $J*P_{2,0}$. Hence $D(A)=1$.

It is natural to begin by investigating the case $D(A)=0$. This turns out to give
a description of quasi-standard C$^*$-algebras.

\begin{lemma}\label{D(A)=0} Let $A$ be a C$^*$-algebra. Then $D(A)=0$ if and only if $A$
is quasi-standard.
\end{lemma}

\begin{proof} Suppose that $A$ is quasi-standard.
By \cite[Theorem 3.3 ((i) $\Rightarrow$ (v))]{ASo},
$\Sub(A)=\Min$-$\Primal(A)$ and every primitive ideal of $A$
contains a unique minimal primal ideal.  Hence $D(A)=0$.

Conversely, suppose that $D(A)=0$ and that $I\in \Sub(A)$. There
exists a minimal primal ideal $J$ of $A$ such that $ J\subseteq I$.
Since $I*J$ and $D(A)=0$, $I=J$. Hence $\Min$-$\Primal(A)$ is
$\tau_s$-closed in $\Primal'(A)$. Furthermore each primitive ideal
of $A$ must contain a unique minimal primal ideal, for otherwise
$D(A)\ge 1$. Hence $A$ is quasi-standard by \cite[Theorem 3.3 ((v)
$\Rightarrow$ (i))]{ASo}.
\end{proof}

\bigskip

 The only motion group $G_N$ for which the C$^*$-algebra is
quasi-standard is $G_2$, and this has been studied in \cite[Example
4.1]{AKSI}.

\medskip

If $I*J$ then there exists $P\in \Prim(A)$ such that $P\supseteq
I+J$, and if $P,Q\in \Prim(A)$ with $P\sim Q$ then there exists
$I\in \Min$-$\Primal(A)\subseteq\Sub(A)$ with $I\subseteq P\cap Q$.
Hence if $I_0*\ldots
*I_k$ is a walk in $\Sub(A)$ of length $k\ge 2$ then there is a walk
$P_1\sim\dots \sim P_k$ of length $k-1$ in $\Prim(A)$, where
$P_i\supseteq I_{i-1}+I_i$. Any strictly shorter walk between $P_1$
and $P_k$ yields a strictly shorter walk between $I_0$ and $I_k$. In
this way one sees that $D(A)\le \Orc(A)+1$, and a similar argument
shows that $\Orc(A)\le D(A)+1$ (and also that $D(A)$ is infinite if
and only if $\Orc(A)$ is infinite). Similarly, for $G\in \Glimm(A)$,
$\{P\in\Prim(A): P\supseteq G\}$ is a $\sim$-connected subset of
$\Prim(A)$ if and only if $\{J\in\Sub(A): J\supseteq G\}$ is a
$*$-connected subset of $\Sub(A)$.

\begin{lemma}\label{Orc(M(A)leqn} Let $A$ be a $\sigma$-unital C$^*$-algebra and let
 $n$ be a positive integer. Suppose that for all $G\in \Glimm(A)$ and
 all $R, S\in \Prim(M(A)/H_G)$, $d_{M(A)}(R, S)\le n$. Then $\Orc(M(A))\le n$.
\end{lemma}

\begin{proof} We may assume that $\Orc(M(A))\geq2$. Let $H$ be a Glimm ideal
of $M(A)$ and suppose that $R,S\in \Prim(M(A)/H)$ with $2\le
d_{M(A)}(R,S)<\infty$. Set $d_{M(A)}(R,S)=k$. Then by
 \cite[Lemma 2.4]{Som} or Lemma~\ref{admissible-chain} (applied to $M(A)$),
 there is an
admissible chain $X_1, \ldots, X_k$ of length $k$ of closed subsets
of $\Prim(M(A))$ such that $R\in X_1\setminus X_2$ and $S\in
X_k\setminus X_{k-1}$. Since $X_1\setminus X_2$ is an open subset of
$\Prim(M(A))$, there exists $b\in M(A)$ such that $\Vert
b\Vert=\Vert b_R\Vert=1$ and $b+T=0$ for all $T\in \Prim(M(A))$ with
$T\notin X_1\setminus X_2$. Similarly, there exists $c\in M(A)$ such
that $\Vert c\Vert=\Vert c_S\Vert=1$ and $c+T=0$ for all $T\in
\Prim(M(A))$ with $T\notin X_k\setminus X_{k-1}$.

Set $V=\{G\in \Glimm(A): \Vert b+H_G\Vert\ge 1/2\}$ and $W=\{G\in
\Glimm(A): \Vert c+H_G\Vert\ge 1/2\}$. Using the canonical
homeomorphism $\iota:\beta\Glimm(A)\to\Glimm(M(A))$ (see Section
$2$) and the upper semi-continuity of norm functions on
$\Glimm(M(A))$, we obtain that $V$ and $W$ are closed subsets of
$\Glimm(A)$. Furthermore, $H$ lies in the closure of both $\iota(V)$
and $\iota(W)$. To see this, let $(P_{\alpha})$ be a net in
$\Prim(A)$ with $\tilde P_{\alpha}\to R$. Set
$G_{\alpha}=\phi_A(P_{\alpha})$. Then, by the continuity of
$\phi_{M(A)}$,
 $$H_{G_{\alpha}}=i(G_{\alpha})=\phi_{M(A)}(\tilde
P_{\alpha})\to\phi_{M(A)}(R)= H.$$
Eventually $\Vert b+\tilde
P_{\alpha}\Vert>1/2$, by lower semi-continuity of norm functions on
$\Prim(M(A))$, and hence eventually $\Vert
b+H_{G_{\alpha}}\Vert\ge\Vert b+\tilde P_{\alpha}\Vert>1/2$. Thus
eventually $G_{\alpha}\in V$, so $H$ lies in the closure of
$\iota(V)$. Similarly $H$ lies in the closure of $\iota(W)$.

Since $A$ is $\sigma$-unital, $\Glimm(A)$ is normal (see Section
$2$)
 and so $V\cap W$ is non-empty by \cite[Lemma 3.1]{AKSI}. Let $G\in V\cap
W$. Then there exists $T \in \Prim(M(A)/H_G)$ such that $\Vert
b+T\Vert>0$, and hence such that $T\in X_1\setminus X_2$. Similarly
there exists $T'\in \Prim(M(A)/H_G)$ such that $T'\in X_k\setminus
X_{k-1}$. But then $d_{M(A)}(T,T')\ge k$ by \cite[Lemma 2.1]{Som}.
Hence $k\le n$, by hypothesis, so $\Orc(M(A))\le n$ as required.
\end{proof}

\bigskip

 We are now ready for the first result linking $\Orc(M(A))$ and
$D(A)$.

\begin{thm}\label{Orc(M(A)leqD(A)+1} Let $A$ be a $\sigma$-unital C$^*$-algebra such
 that $X^1$ is closed whenever $X$ is a closed subset of
$\Prim(A)$. If $D(A)\ge 1$ then
$$\Orc(M(A))\le D(A)+1\le \Orc(A)+2.$$
\end{thm}

\begin{proof} Suppose that $D(A)\geq1$. Since $D(A)\le \Orc(A)+1$, it suffices to show that
$\Orc(M(A))\le D(A)+1$. Without loss of generality we may assume
that $D(A)<\infty$ and hence $\Orc(A)<\infty$. It follows from
Proposition~\ref{phi-closed} that $\phi_A$ is a closed map and that
every Glimm class in $\Prim(A)$ is $\sim$-connected.

Towards a contradiction, suppose that there exist $G \in \Glimm(A)$
and $Q, R \in \Prim(M(A)/H_G)$ such that
$D(A)+1<d_{M(A)}(Q,R)\leq\infty$. Then there exists $k\in\NN$ such
that $D(A)+1<k\leq d_{M(A)}(Q,R)$. Since $D(A)\geq 1$, we have
$k\geq3$.
By \cite[Lemma 2.4]{Som} or Lemma~\ref{admissible-chain} (applied to
$M(A)$), there is a chain $X_1, \ldots, X_k$ on $\Prim(M(A))$ with
$Q\in X_1\setminus X_2$ and $R\in X_k\setminus X_{k-1}$. Set
$W=X_1\setminus X_2$ and let $V=\{ P\in \Prim(A): \tilde P\in W\}$.
Then $W$ is an open subset of $\Prim(M(A))$ containing $Q$ and
$\tilde V=\{ \tilde P: P\in V\}$ is a dense open subset of $W$ by
the density of the open subset $\Prim(A)^{\sim}$ in $\Prim(M(A))$.
We claim that $\overline V$ (the closure of $V$ in $\Prim(A)$) meets
$\Prim(A/G)$. To see this, let $(Q_{\alpha})$ be a net in $V$ such
that $\tilde Q_{\alpha}\to Q$. Using the canonical homeomorphism
$\iota: \beta\Glimm(A)\to\Glimm(M(A))$ and the continuity of
$\phi_{M(A)}$, we have
$$\iota(\phi_{A}(Q_{\alpha}))=\phi_{M(A)}(\tilde Q_{\alpha})
\to \phi_{M(A)}(Q)=H_G=\iota(G).$$  Hence $\phi_A(Q_{\alpha})\to G$.
Thus $G$ belongs to the closure of $\phi_A(V)$ in $\Glimm(A)$. Since
$\phi_A$ is a closed map this implies that there exists $T\in
\overline V$ such that $\phi_A(T)=G$.

Now let $(P_{\alpha})$ be a net in $V$ such that $P_{\alpha}\to T$.
For each $\alpha$, let $I_\alpha\in \Min$-$\Primal(A)$ with
$I_\alpha\subseteq P_{\alpha}$.
By passing to a subnet, if necessary, we may assume that
$(I_{\alpha})$ is $\tau_s$-convergent in ${\rm Id}(A)$, with limit
$I$ say. For $a\in I$,
$$ 0=\Vert a+I\Vert = \lim\Vert a+I_{\alpha}\Vert \geq \liminf\Vert
a+P_{\alpha}\Vert\geq\Vert a+T\Vert.$$
 Thus $I\subseteq T$ and so $I\ne A$
and $I\in \Sub(A)$. Since each proper closed primal ideal of $A$
contains a unique Glimm ideal \cite[Lemma 2.2]{ASo}, and $T\supseteq
I$ and $T\supseteq G$, it follows that $I\supseteq G$.

Set $X_i'=\{P\in \Prim(A): \tilde P\in X_i\}$ $(1\le i\le k)$. Then
$X'_1, \ldots, X'_k$ is a chain on $\Prim(A)$. We claim that
$\Prim(A/I)\subseteq X_1'$. To see this, let $T'\in \Prim(A/I)$.
Suppose that $T'$ does not lie in the closed set $X'_1$ and set
$M=\ker X'_1$. Then $T'\not\supseteq M$ and so $I\not\supseteq M$.
Since $I_{\alpha}\to I$ ($\tau_s$),  eventually
$I_{\alpha}\not\supseteq M$ and so there exists $Q_{\alpha}\in
\Prim(A/I_{\alpha})$ such that $Q_{\alpha}\not\supseteq M$. Hence
$Q_{\alpha}\notin X'_1$. On the other hand  $P_{\alpha}\in
\Prim(A/I_{\alpha})$ and $P_{\alpha}\in V\subseteq X'_1$. Thus
$P_{\alpha}$ and $Q_{\alpha}$ are in disjoint open subsets of
$\Prim(A)$. This contradicts the primality of $I_{\alpha}$, so $T'$
must belong to $X'_1$.

In the same way, using $R\in X_k$, we may obtain $J\in \Sub(A)$ with
$J\supseteq G$ such that $\Prim(A/J)\subseteq X_k'$. Since
$\Prim(A/G)$ is a $\sim$-connected subset of $\Prim(A)$, $I$ and $J$
are in the same $*$-component of $\Sub(A)$. Let
$I=I_0*I_1*\ldots*I_{n-1}
*I_n=J$ be any walk in $\Sub(A)$ from $I$ to $J$, of length $n$.
Then $\Prim(A/I_1)$ meets $X_1'$ and hence $\Prim(A/I_1)\subseteq
X_1'\cup X_2'$, since $\Prim(A/I_1)$ is a limit set and $(X'_1\cup
X'_2)\setminus X'_3$ and $(X'_3\cup\ldots X'_k)\setminus X'_2$ are
disjoint open subsets of $\Prim(A)$. By induction,
$\Prim(A/I_i)\subseteq X_1'\cup X_2'\cup\ldots\cup X_{i+1}'$ for
$1\le i\le k-1$. Thus the least $i$ such that $\Prim(A/I_i)$ can
meet $X_k'$ is $i=k-2$ (recall that $X'_{k-j}\cap X_k'$ is empty for
$2\le j\le k-1$). But $\Prim(A/I_{n-1})$ does meet $\Prim(A/J)$ (and
hence $X'_k$) and thus $n-1 \geq k-2$. Hence $d^*(I,J)\ge k-1$ and
so $D(A)\ge k-1>D(A)$, a contradiction.

We have shown that for all $G\in \Glimm(A)$ and $R, S\in
\Prim((M(A)/H_G)$, $d_{M(A)}(R, S)\le D(A)+1$. It follows from
Lemma~\ref{Orc(M(A)leqn} that $\Orc(M(A))\leq D(A)+1$.
\end{proof}

\bigskip

 If $A$ is a $\sigma$-unital $C^*$-algebra with
$D(A)=0$ (i.e. $A$ quasi-standard) then $\Orc(M(A))\leq2$
\cite[Theorem 3.4]{AKSI}. Combining this with
Theorem~\ref{Orc(M(A)leqD(A)+1}, we see that if $A$ is a
$\sigma$-unital C$^*$-algebra such that $X^1$ is closed whenever $X$
is a closed subset of $\Prim(A)$ then $\Orc(M(A))\le \Orc(A)+2$.

\begin{cor}
Let $A$ be a $\sigma$-unital $C^*$-algebra such that $X^1$ is closed
whenever $X$ is a closed subset of $\Prim(A)$. Then
$$\Orc(A)\leq \Orc(M(A))\leq \Orc(A)+2.$$
\end{cor}

\begin{proof}
By Theorem~\ref{Orc(M(A))geqOrc(A)}, $\Orc(A)\leq \Orc(M(A))$. The
second inequality has been noted above.
\end{proof}

\bigskip

 In the second part of this section, we shall consider
$C^*$-algebras with the following properties:

(1) $X^1$ is closed whenever $X$ is a closed subset of $\Prim(A)$;

(2) there is a dense subset $S\subseteq \Prim(A)$ such that: (a)
each $P\in S$ is a Glimm ideal; (b) each $P\in S$ is a maximal
ideal; (c) $A/P$ is non-unital for all $P\in S$.

\medskip

\noindent Property (1) has been discussed earlier, at the beginning
of Section $4$. Property (2) is a considerable restriction, but it
is satisfied by the $C^*$-algebras of many locally compact groups
including $SL(2,\R)$, $SL(2,\C)$, the motion groups $G_N$ (see
below) and all non-abelian, connected, simply connected nilpotent
Lie groups.

If $\Prim(A)$ is a $T_1$-space, $\Orc(A)<\infty$, and $A$ has
property (1) then every Glimm class is $\sim$-connected by
Proposition~\ref{phi-closed}, so every separated point of $\Prim(A)$
automatically satisfies (a) and (b) of property (2).
Hence if $A$ is a separable unital C$^*$-algebra with $\Prim(A)$ a
$T_1$-space and $\Orc(A)<\infty$ then its stabilization $A\otimes
\mathcal{K}$ satisfies properties (1) and (2) (where $\mathcal{K}$
is the $C^*$-algebra of compact linear operators on a separable
Hilbert space of infinite dimension). Also, if $A$ is a separable,
liminal $C$*-algebra such that $\Orc(A)<\infty$ and the set $S_1=\{P
\in \Prim(A): A/P \quad\mbox{is non-unital}\}$ is dense in
$\Prim(A)$, then the assumption of property (1) automatically leads
to property (2). To see this note that $S_1$ is a $G_\delta$ subset
of $\Prim(A)$ (see the proof of \cite[Proposition 12]{Del}). On the
other hand, the set $S_2$ of separated points of $\Prim(A)$ is also
a dense $G_\delta$ subset \cite[3.9.4(b)]{Dix}. As above, each $P$
in $S_2$ satisfies (a) and (b). Since $\Prim(A)$ is a Baire space,
the set $S=S_1 \cap S_2$ has the required properties. Incidentally,
the density of $S_1$ always forces $Z(A)=\{0\}$ and the converse is
true when $A$ is liminal \cite[Proposition 12]{Del}.

It follows from property (2) that every $P\in S$ is a separated
point of $\Prim(A)$.
If a net converges in $\Prim(A)$  to a separated point $Q$ then
every cluster point must contain $Q$ and hence the net converges to
$Q$ in $\tau_s$ \cite[Theorem 2.1]{Fe}. It follows that $S$ is
$\tau_s$-dense in the set of separated points of $\Prim(A)$. Thus,
if $A$ is also separable, every ideal in $\Sub(A)$ is the
$\tau_s$-limit of a sequence in $S$ (cf. the discussion preceding
Theorem~\ref{M(A)and(n+1)}). Now suppose that $I\in \Sub(A)$, that
$(P_n)$ is a sequence in $S$ such that $P_n \to I$ ($\tau_s$) and
that a primitive ideal $P$ lies in the closure of $\{P_n:n\geq1\}$
but $P\neq P_n$ for all $n$. Let $a\in I$ and $\epsilon>0$. Since
$P_n \to I$ ($\tau_s$), there exists $N\in\NN$ such that $\Vert
a+P_n\Vert<\epsilon$ for all $n> N$. As $P$ does not belong to the
closed set $\{P_1,\ldots,P_N\}$ it belongs to the closure of the set
$\{P_n:n>N\}$. By lower semi-continuity of norm functions on
$\Prim(A)$, $\Vert a+P\Vert \leq \epsilon$. Since $\epsilon$ was
arbitrary, we obtain that $a\in P$ and hence that $P\supseteq I$.
We
shall use this fact several times in the proof of
Theorem~\ref{Orc(M(A))=D(A)+1}.

We now show that if $A=C^*(G_N)$ ($N\geq2$) then $A$ satisfies both
properties (1) and (2). Property (1) has been previously observed as
a consequence of Lemma~\ref{X^1closed}. The set $S=\{\ker \pi:
\pi\in\mathcal{U}_N\}$ is a dense open subset of separated points of
$\Prim(A)$ which are also Glimm ideals (see Section $3$). For $P\in
S$, $A/P$ is $^*$-isomorphic to the non-unital, simple $C^*$-algebra
$\mathcal{K}$ and so $S$ satisfies 2(b) and 2(c).

\bigskip

 Next we need a technical lemma.

\begin{lemma}\label{lemma-for-(1)and(2)} Let $A$ be a separable C$^*$-algebra with $\Orc(A)<\infty$ and
with properties (1) and (2) above. Let $G\in \Glimm(A)$ and let
$I\in \Sub(A)$ with $G\subseteq I$. Set $X_G=\{P\in \Prim(A):
P\supseteq G\}$. Let $V$ be an open subset of $\Prim(A)\setminus
X_G$ and suppose that there is a sequence $(P_n)_{n\ge 1}$ in $
S\cap V$ such that $P_n\to I$ $(\tau_s)$. Then there exists $Q\in
\Prim(M(A))$ with $Q\supseteq H_G$ such that
\begin{enumerate}
\item[{\rm (i)}] $Q$ lies in the closure of $\{ \tilde P_n:n\ge
1\}$;
\item[{\rm (ii)}] if $(Q_{\alpha})$ is any net in $\Prim(A)$
with $\tilde Q_{\alpha}\to Q$ then eventually $Q_{\alpha}\in V$;
\item[{\rm (iii)}] $\{Q\}^1$ lies in the closure of $\tilde V=\{\tilde
P:P\in V\}$ in $\Prim(M(A))$.
\end{enumerate}
\end{lemma}

\begin{proof} Set $W=\Prim(A)\setminus V$ and
$U=\Glimm(A)\setminus\phi_A(W)$. By Proposition~\ref{phi-closed},
$\phi_A$ is closed and so $U$ is an open subset of $\Glimm(A)$ and
$G\notin U$. On the other hand, $P_n=\phi_A(P_n)\in U$ for each $n$
since $\{P_n\}$ is a Glimm class. Also, for any $P\in\Prim(A/I)$,
$P_n\to P$ in $\Prim(A)$ and so $P_n=\phi_A(P_n)\to\phi(P)= G$ in
$\Glimm(A)$. Hence $G$ lies in the boundary of the open set $U$.
Since $A$ is separable, it follows from \cite[Lemma 3.9]{ASomSS}
that $U$ is the cozero set of a continuous real-valued function $f$
on $\Glimm(A)$. Replacing $f$ by $|f|/(1+|f|)$, we may assume that
$0\le f\le 1$.

By \cite[Theorem 2.2 (iii),(iv)]{AKSI}, there exists an element
$b\in M(A)$ with $0\leq b\leq 1$ such that $(1-b)\in \tilde K$ for
all $K\in \Glimm(A)\setminus U$ and $\Vert (1-b)+\tilde K\Vert=1$
for all $K\in U$ such that $A/K$ is non-unital. In particular,
$\Vert (1-b)+\tilde P_n\Vert=1$ for all $n\ge 1$. The set
$$N=\{R\in
\Prim(M(A)):\Vert (1-b)+R\Vert\ge 1\}$$
 is compact and hence the
sequence $(\tilde P_n)$ in $N$ has a convergent subnet $(\tilde
P_{n_{\alpha}})$ with a limit $Q\in N$. Since $\Glimm(M(A))$ is
Hausdorff,
$$\phi_{M(A)}(Q) =\lim\phi_{M(A)} (\tilde P_{n_{\alpha}})
=\lim\iota(\phi_A(P_{n_{\alpha}}))= \iota(G)=H_G$$ and so
$Q\supseteq H_G$.

Set $Y=\{ R\in \Prim(M(A)): \Vert (1-b)+R\Vert>0\}$. Then $Y$ is a
neighbourhood of $Q$ in $\Prim(M(A))$. If $P\in \Prim(A)$ and
$K:=\phi_A(P)\notin U$ then $1-b\in \tilde{K}\subseteq\tilde{P}$ and
so $\tilde{P}\notin Y$. For (ii), suppose that
$\tilde{Q}_{\alpha}\to Q$. Then eventually $\tilde{Q}_{\alpha}\in Y$
and so $Q_{\alpha}\in \phi_A^{-1}(U)\subseteq V$.

Finally, suppose that $T\in \Prim(M(A))$ with $T\sim Q$. By the
density of $\{\tilde{P}:P\in\Prim(A)\}$ in $\Prim(M(A))$ there
exists a net $(Q_{\alpha})$ in $\Prim(A)$ such that $\tilde
Q_{\alpha}\to Q, T$. By (ii), eventually $Q_{\alpha}\in V$. Hence
$T$ lies in the closure of $\tilde V$ in $\Prim(M(A))$.
\end{proof}


\bigskip

 We come now to the main theorem of this section. In the course of
this, we shall need the fact that if $A$ is a $C^*$-algebra and $G$
is a $\sigma$-unital Glimm ideal of $A$ then
$\Glimm(A)\setminus\{G\}$ is a normal subspace of $\Glimm(A)$. To
see this, first note that the complete regularity of the Hausdorff
space $\Glimm(A)$ passes to any subspace. Secondly,
$\Glimm(A)\setminus\{G\}$ is the image under $\phi_A$ of the set
$\{P\in\Prim(A): P\not\supseteq G\}$ which is homeomorphic to
$\Prim(G)$ and therefore $\sigma$-compact. Thus
$\Glimm(A)\setminus\{G\}$ is a $\sigma$-compact, completely regular
Hausdorff space and is therefore normal (see \cite[3D]{GJ} or
\cite[Ch.2, Proposition 1.6]{Pears}). If $A$ is separable, this
applies to any $G\in \Glimm(A)$. Incidentally, it is always the case
that $\Glimm(A)\setminus\{G\}$ is canonically homeomorphic to
$\Glimm(G)$ but we have avoided the need to prove that here.

\bigskip

\begin{thm}\label{Orc(M(A))=D(A)+1} Let $A$ be separable C$^*$-algebra having properties (1) and
(2) above, and with $D(A)\ge 1$. Then $\Orc(M(A))= D(A)+1$.
\end{thm}

\begin{proof}
By Theorem~\ref{Orc(M(A)leqD(A)+1}, $\Orc(M(A))\le D(A)+1$. Thus it
suffices to show that $\Orc(M(A))\ge D(A)+1$ and for this we may
assume that $\Orc(M(A))<\infty$. Then it follows from
Theorem~\ref{Orc(M(A))geqOrc(A)} that $\Orc(A)<\infty$ (and so
$D(A)<\infty$) and from Proposition~\ref{phi-closed} that $\phi_A$
is a closed map.

We begin by considering the case $D(A)=1$. Suppose that
$\Orc(M(A))=1$. By \cite[Corollary 2.4]{AKSI} and property (2),
$\Prim(A)$ is discrete and so $D(A)=0$, a contradiction. Thus
$\Orc(M(A))\ge 2$, as required.

Now suppose that $D(A)=2$. Let $I, J\in \Sub(A)$ with $d^*(I,J)=2$.
Let $G$ be the unique Glimm ideal such that $I\cap J\supseteq G$ and
set $X_G=\Prim(A/G)$. Set $X=\Prim(A/I)$ and $Y=\Prim(A/J)$. Then
$X$ and $Y$ are disjoint closed subsets of $X_G$. Since $X_G$ is not
a singleton, $G\notin S$ and $X_G\cap S=\emptyset$. Let $(P_n)$ and
$(Q_n)$ be sequences in $S$ with $P_n\to I$ $(\tau_s)$ and $Q_n\to
J$ $(\tau_s)$. Since $I\neq J$ and $\tau_s$ is Hausdorff, we may
assume that the sets $X'= \{ P_n: n\ge 1\}$ and $Y'=\{ Q_n:n\ge
1\})$ are disjoint. By an observation preceding
Lemma~\ref{lemma-for-(1)and(2)}, $X'\cup X$ and $Y'\cup Y$ are
closed subsets of $\Prim(A)$. Since $\phi_A$ is closed, $\phi_A(X')$
and $\phi_A(Y')$ are disjoint closed subsets of the normal space
$\Glimm(A)\setminus \{G\}$. By two applications of normality, there
exist disjoint open subsets $U'$ and $V'$ of $\Glimm(A)\setminus
\{G\}$ such that $X'=\phi_A(X')\subseteq U'$ and
$Y'=\phi_A(Y')\subseteq V'$ and such that the closures of $U'$ and
$V'$ in $\Glimm(A)\setminus\{G\}$ are disjoint. Set
$U''=\phi_A^{-1}(U')$ and $V''=\phi_A^{-1}(V')$.

If $P,Q\in\Prim(A)\setminus X$ and $P\sim Q$ relative to the space
$\Prim(A)\setminus X$ then $P\sim Q$ in $\Prim(A)$ and hence
$P\approx Q$ in $\Prim(A)$. Hence $X'$ and $X_G\setminus X$ are
relatively closed and $\sim$-saturated subsets of the space
$\Prim(A)\setminus X$ which is homeomorphic to $\Prim(I)$. Since
$I$ is a separable C$^*$-algebra, $\Prim(I)$ is $\sigma$-compact and
so it follows from Lemma~\ref{Lindelof-lemma} that there exist
disjoint open sets $U'''$ and $W$ in $\Prim(A)\setminus X$ such that
$X'\subseteq U'''$ and $X_G\setminus X\subseteq W$. Similarly there
exist disjoint open sets $V'''$ and $W'$ in $\Prim(A)\setminus Y$
such that $Y'\subseteq V'''$ and $X_G\setminus Y\subseteq W'$. Set
$U=U''\cap U'''$ and $V=V''\cap V'''$. Then $U$ and $V$ are open
subsets of $\Prim(A)$ with $P_n\in U$, and $Q_n\in V$ for all $n$.
Since $U\subseteq \phi_A^{-1}(U') \subseteq
\phi_A^{-1}(\Glimm(A)\setminus \{G\})$, we have $U\cap
X_G=\emptyset$ and similarly $V\cap X_G=\emptyset$.
 It follows from
Lemma~\ref{lemma-for-(1)and(2)} that there exist $R,T\in \Prim(M(A))$
with $R,T\supseteq H_G$ such that $\{R\}^1$ lies in the closure of
$\tilde U$ in $\Prim(M(A))$ and $\{T\}^1$ lies in the closure of
$\tilde V$.

Suppose that $P\in\overline U\cap \overline V$ (where the closures
are taken in $\Prim(A)$). Then $\phi_A(P) \in \overline{U'}\cap
\overline{V'} =\{G\}$ (where the closures are taken in $\Glimm(A)$)
and so $P\supseteq G$. On the other hand, by the properties of
$U'''$ and $W$, $P\notin X_G\setminus X$ and similarly $P\notin
X_G\setminus Y$. Since $X$ and $Y$ are disjoint subsets of $X_G$, we
have $P\notin X_G$, a contradiction. Thus $\overline U\cap \overline
V=\emptyset$ and hence the closures of $\tilde U$ and $\tilde V$ in
$\Prim(M(A))$ are disjoint by Lemma~\ref{disjoint-prim}. Hence
$\{R\}^1$ and $\{T\}^1$ are disjoint and thus $d_{M(A)}(R,T)\ge 3$.
Since $\Orc(M(A))<\infty$, Glimm classes are $\sim$-connected in
$\Prim(M(A))$ \cite[Corollary 2.7]{Som}, and hence $\Orc(M(A))\ge 3$
as required in this case.

Finally suppose that $D(A)=k\ge 3$. Let $I, J\in \Sub(A)$ with and
$d^*(I,J)=k$ and let $G$ be the unique Glimm ideal such that $I\cap
J\supseteq G$. Set $X=\Prim(A/I)$ and $Y=\Prim(A/J)$. Then $X$ and
$Y$ are disjoint closed subsets of $X_G:=\Prim(A/G)$,
$d_A(X,Y)=k-1$, $G\notin S$ and $X_G\cap S=\emptyset$. The sets $X$
and $Y$ are $\sim$-connected since $I$ and $J$ are primal, and hence
$X\cup Y$ is $\sim$-connected since $d_A(X,Y)<\infty$. By
Lemma~\ref{admissible-chain}, there is an admissible chain $X_1,
\ldots, X_{k-1}$ of closed subsets of $\Prim(A)$ with $X\subseteq
X_1\setminus X_2$ and $Y\subseteq X_{k-1}\setminus X_{k-2}$. For
$1\le i\le k-1$, let $Z_i$ be the closure of $\tilde X_i$ in
$\Prim(M(A))$. Then as in the proof of
Theorem~\ref{Orc(M(A))geqOrc(A)} we have that $Z_1, \ldots, Z_{k-1}$
is a chain on $\Prim(M(A))$ of length $k-1$.

Let $(P_n)$ be a sequence in the dense subset $S$ with $P_n\to I$
$(\tau_s)$. Since $X_1\setminus X_2$ is an open set containing
$\Prim(A/I)$ we may assume that $P_n\in X_1\setminus X_2$ for all
$n$. Set $X'=\{ P_n: n\ge 1\}$.
Then $X'$ and $X_{2}$ are disjoint. As before, $X'$ is a relatively
closed and $\sim$-saturated subset of the space $\Prim(A)\setminus
X$ which is homeomorphic to the $\sigma$-compact space $\Prim(I)$.
Let  $E=((X_1\cap X_G)\setminus X)\cup X_2\cup\ldots\cup X_{k-1}$, a
relatively closed subset of $\Prim(A)\setminus X$ disjoint from
$X'$. So $E=F\cap(\Prim(A)\setminus X))$ for some closed subset $F$
of $\Prim(A)$ disjoint from $X'$. Since $X'$ is a $\sim$-saturated
subset of $\Prim(A)$, $F^1$ is disjoint from $X'$ and is closed in
$\Prim(A)$ by property (1). Now suppose that $P\in\overline{E^1}$
(where both operations are taken relative to the space
$\Prim(A)\setminus X$). Then, in $\Prim(A)$, $P\in
\overline{F^1}=F^1$ and so $P\notin X'$.
 Hence by Lemma~\ref{Lindelof-lemma} there exist
disjoint open subsets $U$ and $V$ of $\Prim(A)\setminus X$ such that
$X'\subseteq U$ and $E=((X_1\cap X_G)\setminus X)\cup
X_2\cup\ldots\cup X_{k-1}\subseteq V$.

By Lemma~\ref{lemma-for-(1)and(2)} there exists $R\in \Prim(M(A))$
such that $R\supseteq H_G$ and $\{R\}^1$ is contained in the closure
of $\tilde{U}$ in $\Prim(M(A))$. But $\overline{U}\subseteq
\Prim(A)\setminus V$ (where $\overline U$ denotes the closure of $U$
in $\Prim(A)$), so $\overline{U}\cap (X_2\cup\ldots\cup X_{k-1})$ is
empty. Hence, by Lemma~\ref{disjoint-prim}, $\{R\}^1$ does not meet
$Z_2\cup\ldots\cup Z_{k-1}$. Similarly there exists $T\in
\Prim(M(A))$ such that $T\supseteq H_G$ and $\{T\}^1$ does not meet
$Z_1\cup\ldots\cup Z_{k-2}$. Since $Z_1, \ldots, Z_{k-1}$ is a chain
on $\Prim(M(A))$, $\{R\}^1\subseteq Z_1\setminus Z_2$. But
$Z_1\setminus Z_2$ and $\Prim(M(A))\setminus Z_1$ are disjoint open
subsets of $\Prim(M(A))$ and so $\{R\}^2\subseteq Z_1$. Proceeding
inductively, we obtain that $\{R\}^i\subseteq Z_{1}\cup\ldots\cup
Z_{i-1}$ for $2\le i\le k-1$.  Thus
$\{R\}^{k-1}\cap\{T\}^1=\emptyset$ and so $d_{M(A)}(R, T)\ge k+1$.
Since $\Orc(M(A)<\infty$, $\Orc(M(A))\ge k+1$ as required.
\end{proof}

\section{The cases $N\equiv1$ and $N\equiv3 \mod 4$.}

In this section, we finally complete the computation of
$K(M(C^*(G_N)))$. The next result will be used in Theorems
\ref{Ncong1} and \ref{Ncong3} to show that if $A=C^*(G_N)$ and $N$
is odd then
$$K(M(A))=K_s(M(A))=\frac{1}{2}{\rm Orc}(M(A))=\frac{N+1}{4}.$$

\begin{prop}\label{D-motion} Let $A=C^*(G_N)$ with $N$ odd. Then $D(A)\geq
\frac{N-1}{2}$.
\end{prop}

\begin{proof} Let $N=2k+1$ and suppose first of all that $k>1$.
 Since $SO(N)^{\wedge}$ is a $\sim$-connected subset of
 $\widehat{G_N}$ (Theorem~\ref{Orc(G_N)}), the
 ideals
 $I_{0,(0,\ldots,0)}$ and $I_{0,(1,\ldots,1)}$
 belong to the same $*$-component of ${\rm Sub}(A)$ (in fact, it is
 easy to exhibit a $*$-walk between them via the
 ideals $I_{0,(1,\ldots,1,0\ldots,0)}$).

 Suppose that $0\leq
i\leq k-2$ and that $\sigma=(p_1, \ldots,p_k)$ and $\sigma'=(p'_1,
\ldots,p'_k)$ are elements of $SO(N-1)^{\wedge}$ such that
$I_{0,\sigma}*I_{0,\sigma'}$ and $p_j=0$ for $i<j\leq k$. Then there
exists $\pi\in SO(N)^{\wedge}$ such that $\pi|_{SO(N-1)}\geq \sigma$
and $\pi|_{SO(N-1)}\geq \sigma'$. So there exist integers $m_1\geq
m_2\geq\ldots\geq m_k\geq0$ such that
$$m_1\geq p_1\geq m_2\geq\ldots\geq p_i\geq m_{i+1}\geq0 \geq
m_{i+2}$$ and
$$m_1\geq p'_1\geq \ldots\geq m_{i+2}\geq p'_{i+2}\geq\ldots\geq
m_k\geq p'_k\geq -m_k.$$ Thus $m_{i+2}=\ldots=m_k=0$ and so
$p'_{i+2}=\ldots=p'_k=0$. It follows that
$$d^*(I_{0,(0,\ldots,0)},I_{0,(1,\ldots,1)})\geq k$$ and so $D(A)\geq
k$ as required.

  In the case $k=1$, consideration of the inequalities $m_1\geq 1\geq -m_1$ and
  $m_1\geq0\geq -m_1$ shows that the hull of $I_{0,(1)}$ is strictly contained
  in the hull of $I_{0,(0)}$. Thus $I_{0,(1)}$ strictly contains $I_{0,(0)}$
  and so $D(A)\neq0$.
  \end{proof}

\medskip

Our final two results are similar in nature to Theorems \ref{Ncong0}
and \ref{Ncong2} but the proofs differ from these in the following
respects. Firstly, since $N$ is now odd, the checking of the
hypotheses of Theorems \ref{M(A)and(n+1)} and \ref{M(A)and(n+3/2)}
is somewhat different and so we give the details. Secondly, we use
Proposition~\ref{D-motion} and Theorem~\ref{Orc(M(A))=D(A)+1} in
place of Corollary~\ref{Orc(M(A))motion}.

\begin{thm}\label{Ncong3} Let $A=C^*(G_N)$ where $N\equiv 3$ $({\rm mod}\ 4)$. Then
$$K(M(A))=K_s(M(A))=\frac{1}{2}{\rm Orc}(M(A))=\frac{N+1}{4}.$$
\end{thm}
\begin{proof} Let $k=\frac{N-1}{2}$ and $m=\frac{N-3}{4}$ so that $N=2k+1$
and $k=2m+1$.  We begin with the case $m\neq0$ ($N\geq7$). We aim to
show that $A$ satisfies the hypotheses of Theorem~\ref{M(A)and(n+1)}
with $n=m$. Let $I^{(1)},I^{(2)},I^{(3)}\in \Sub(A)$ with
$I^{(1)},I^{(2)},I^{(3)}$ all containing the same Glimm ideal $G$.
If $G$ is one of the separated points of $\Prim(A)$ then
$I^{(1)}=I^{(2)}=I^{(3)}=G$ and we may take
$S^{(1)}=S^{(2)}=S^{(3)}=T=G$.

We may suppose therefore that $G=I_0$ (see Theorem~\ref{Orc(G_N)})
and that $I^{(1)}=I_{0, \sigma}$, $I^{(2)}=I_{0, \sigma'}$, and
$I^{(3)}=I_{0, \sigma''}$ for some $\sigma, \sigma', \sigma''\in
SO(N-1)^{\wedge}$ by Proposition~\ref{motion-sub}. Let $\sigma=(p_1,
\ldots, p_k)$, $\sigma'=(p'_1,\ldots, p'_k)$, and
$\sigma''=(p''_1,\ldots, p''_k)$. Set $P_i=\max\{ p_i, p'_i,
p''_i\}$ for $1\le i\le m+1$. Define $\pi= (P_1, p_1, p_2, \ldots,
p_{k-1})\in SO(N)^{\wedge}$ and similarly $\pi'$ and $\pi''$
(replacing the $p_j$ by $p_j'$ and $p_j''$ respectively). Since
$\pi|_{SO(N-1)} \geq \sigma$, $\ker\pi\supseteq I^1$. Similarly,
$\ker\pi'\supseteq I^2$ and $\ker\pi''\supseteq I^3$.

There is an $m$-step $\sim$-walk in $\widehat{G_N}$ from each of
$\pi$, $\pi'$, $\pi''$ to $\rho:=(P_1,P_2,\ldots , P_{m+1}, 0,\ldots
, 0) \in SO(N)^{\wedge}$. To see this, note that for $1\le i\le
m-1$,
$$(P_1, \ldots , P_i, p_i, p_{i+1}, \ldots , p_{k-i}, 0, \ldots, 0)\sim
(P_1, \ldots, P_i, P_{i+1}, p_{i+1}, \ldots, p_{k-1-i}, 0, \ldots ,
0)$$ because the restrictions to $SO(N-1)$ contain $(P_1,\ldots,
P_i, p_{i+1}, \ldots, p_{k-i}, 0,\ldots, 0),$ and finally
$$(P_1, \ldots, P_m, p_m, p_{m+1}, 0, \ldots, 0) \sim
(P_1, \ldots , P_{m+1}, 0, \ldots, 0)=\rho$$ because the
restrictions to $SO(N-1)$ contain $(P_1, \ldots, P_m, p_{m+1}, 0,
\ldots, 0)$. Similar arguments apply to $\pi'$ and $\pi''$,
replacing the $p_j$ by $p_j'$ and $p_j''$ (respectively).

Taking $S^{(1)}=\ker\pi$, $S^{(2)}=\ker\pi'$, $S^{(3)}=\ker\pi''$
and $T=\ker\rho$, we have satisfied the hypotheses of
Theorem~\ref{M(A)and(n+1)}. Thus $K(M(A))\le m+1=\frac{N+1}{4}$.
Combining this with Proposition~\ref{D-motion},
Theorem~\ref{Orc(M(A))=D(A)+1} and \cite[Theorem 4.4]{Som}, we have
$$K(M(A))\le \frac{N+1}{4}\leq \frac{D(A)+1}{2}
= \frac{{\rm Orc}(M(A))}{2}=K_s(M(A))\leq K(M(A))$$
 and hence equality throughout.

 For the simpler case $N=3$, it suffices to check the hypothesis of
 Theorem~\ref{M(A)and(n+1)} for $G=I_0$ (the only non-maximal Glimm ideal of $A$).
 So let
 $I^{(1)}=I_{0,(q)}$, $I^{(2)}=I_{0,(q')}$ and $I^{(3)}=I_{0,(q'')}$ where
 $(q),(q'),(q'')\in SO(2)^{\wedge}$. Let $Q=\max\{|q|,|q'|,|q''|\}$ and
 let $\pi=(Q)\in SO(3)^{\wedge}$. Then $\pi|_{SO(2)}$ contains
 $(q)$,$(q')$ and $(q'')$ and so $\ker\pi\supseteq I^{(1)}+I^{(2)}+I^{(3)}$. Using
 Theorem~\ref{M(A)and(n+1)} and arguing as above, we have
 $$K(M(A))\le 1\leq \frac{D(A)+1}{2}
= \frac{{\rm Orc}(M(A))}{2}=K_s(M(A))\leq K(M(A))$$
 and hence equality throughout.
 \end{proof}

\begin{thm}\label{Ncong1} Let $A=C^*(G_N)$ where $N\equiv 1$ $({\rm mod}\ 4)$ and
$N\geq5$. Then
$$K(M(A))=K_s(M(A))=\frac{1}{2}{\rm Orc}(M(A))=\frac{N+1}{4}.$$
\end{thm}

\begin{proof} Let $k=\frac{N-1}{2}$ and $m=\frac{N-1}{4}$, so that $N=2k+1$
and $k=2m$. We begin with the case $N\geq9$ so that $m>1$. We aim to
show that $A$ satisfies the hypotheses of
Theorem~\ref{M(A)and(n+3/2)} with $n=m-1$. Let
$I^{(1)},I^{(2)},I^{(3)}\in \Sub(A)$ with $I^{(1)},I^{(2)},I^{(3)}$
all containing the same Glimm ideal $G$. If $G$ is one of the
separated points of $\Prim(A)$ then $I^{(1)}=I^{(2)}=I^{(3)}=G$ and
we may take $S^{(i)}=T^{(i)}=G$ $(1\leq i\leq3)$ so that
$T^{(1)}\cap T^{(2)}\cap T^{(3)}=G$ which is primal.

We may suppose therefore that $G=I_0$ (see Theorem~\ref{Orc(G_N)})
and that $I^{(1)}=I_{0, \sigma}$, $I^{(2)}=I_{0, \sigma'}$, and
$I^{(3)}=I_{0, \sigma''}$ for some $\sigma, \sigma', \sigma''\in
SO(N-1)^{\wedge}$ by Proposition~\ref{motion-sub}. Let $\sigma=(p_1,
\ldots, p_k)$, $\sigma'=(p'_1,\ldots, p'_k)$, and
$\sigma''=(p''_1,\ldots, p''_k)$. Set $P_i=\max\{ p_i, p'_i,
p''_i\}$ for $1\le i\le m$. Define $\pi= (P_1, p_1, p_2, \ldots,
p_{k-1})\in SO(N)^{\wedge}$ and similarly $\pi'$ and $\pi''$
(replacing the $p_j$ by $p_j'$ and $p_j''$ respectively). Since
$\pi|_{SO(N-1)} \geq \sigma$, $\ker\pi\supseteq I^{(1)}$. Similarly,
$\ker\pi'\supseteq I^{(2)}$ and $\ker\pi''\supseteq I^{(3)}$.

There is an $(m-1)$-step $\sim$-walk in $\widehat{G_N}$ from $\pi$
to $\rho:=(P_1,\ldots , P_m,p_{m+1}, 0,\ldots , 0)\in
SO(N)^{\wedge}$. To see this, note that for $1\le i\le m-1$,
$$(P_1, \ldots , P_i, p_i, \ldots , p_{k-i}, 0, \ldots, 0)\sim
(P_1, \ldots, P_i, P_{i+1}, p_{i+1}, \ldots, p_{k-1-i}, 0, \ldots ,
0)$$ because the restrictions to $SO(N-1)$ contain $(P_1,\ldots,
P_i, p_{i+1}, \ldots, p_{k-i}, 0,\ldots, 0)$. Similarly, there is an
$(m-1)$-step $\sim$-walk from $\pi'$ to $\rho':=(P_1,\ldots ,
P_m,p'_{m+1}, 0,\ldots , 0)$ and from $\pi''$ to
$\rho'':=(P_1,\ldots , P_m,p''_{m+1}, 0,\ldots , 0)$.

  The restrictions of $\rho$, $\rho'$ and $\rho''$ to
  $SO(N-1)$ contain $(P_1,\ldots,P_m,0,\ldots,0)$ as a
  common subrepresentation and so $\ker\rho\cap\ker\rho'\cap\ker\rho''$ is a primal ideal of
  $A$, as in the proof of Theorem~\ref{Ncong2}.

Taking $S^{(1)}=\ker\pi$, $S^{(2)}=\ker\pi'$, $S^{(3)}=\ker\pi''$
and
 $T^{(1)}=\ker\rho$, $T^{(2)}=\ker\rho'$, $T^{(3)}=\ker\rho''$, we have satisfied
 the hypotheses of Theorem~\ref{M(A)and(n+3/2)}. Thus $K(M(A))\le
(m-1)+3/2=\frac{N+1}{4}$. Combining this with
Proposition~\ref{D-motion}, Theorem~\ref{Orc(M(A))=D(A)+1} (see the
remark preceding Lemma~\ref{lemma-for-(1)and(2)}) and \cite[Theorem
4.4]{Som}, we have
$$K(M(A))\le \frac{N+1}{4}\leq \frac{D(A)+1}{2}
= \frac{{\rm Orc}(M(A))}{2}=K_s(M(A))\leq K(M(A))$$
 and hence equality throughout.

 For the simpler case $N=5$, it again suffices to check the hypothesis of
 Theorem~\ref{M(A)and(n+3/2)}
 for the Glimm ideal  $G=I_0$. With notation as above, we have that
 the restrictions to $SO(4)$ of $\pi=(P_1,p_1,)$,
 $\pi'=(P_1,p'_1,)$ and $\pi''=(P_1,p''_1,)$ contain $(P_1,0)$ as a
 common subrepresentation and so the ideal $\ker\pi\cap\ker\pi'\cap\ker\pi''$
 is primal. Applying Theorem~\ref{M(A)and(n+3/2)} in the case $n=0$, we have
$$K(M(A))\le \frac{3}{2}\leq \frac{D(A)+1}{2}
=\frac{{\rm Orc}(M(A))}{2}=K_s(M(A))\leq K(M(A))$$
 and hence equality throughout.
 \end{proof}

In the case when $N$ is odd, it can be shown by direct arguments
that $D(A)\leq \frac{N-1}{2}$ (so that $D(A)= \frac{N-1}{2}$).
However, this inequality can also be obtained indirectly from
Theorem~\ref{Orc(M(A))=D(A)+1} and the fact that ${\rm
Orc}(M(A))=\frac{N+1}{2}$ (Theorems \ref{Ncong1}
 and \ref{Ncong3}). In the case when
$N$ is even, direct arguments show that $D(A)=\frac{N}{2}-1$. It
follows that Theorems \ref{Ncong0} and \ref{Ncong2} could be proved
by using Theorem~\ref{Orc(M(A))=D(A)+1} rather than the more
elementary Corollary~\ref{Orc(M(A))motion}. However, since
Theorem~\ref{Orc(M(A))=D(A)+1} uses
Theorem~\ref{Orc(M(A))geqOrc(A)}, no saving would be gained by
adopting this more complicated approach.

\bigskip
\bigskip

\end{document}